\documentclass[11pt,twoside,a4paper]{article}

\usepackage{amsfonts}
\usepackage{amsmath}
\usepackage{amsthm}
\usepackage{paralist}
\usepackage{graphics}
\usepackage{epsfig}
\usepackage{graphicx}
\usepackage[colorlinks=true]{hyperref}
\hypersetup{urlcolor=blue, citecolor=red}
\usepackage{caption}
\usepackage{subcaption}
\usepackage{mathtools}
\usepackage{enumitem}
\usepackage{psfrag}

\newtheorem{theorem}{Theorem}[section]

\newtheorem{lemma}[theorem]{Lemma}
\newtheorem{proposition}[theorem]{Proposition}

\theoremstyle{definition}
\newtheorem{remark}[theorem]{Remark}
\newtheorem{definition}[theorem]{Definition}
\newtheorem{example}[theorem]{Example}

\newcommand{\R}{{\mathbb{R}}}

\newcommand{\C}[1]{\mathbf{C^{#1}}}
\renewcommand{\L}[1]{\mathbf{L^{#1}}}
\newcommand{\Lloc}[1]{\mathbf{L^{#1}_{loc}}}

\newcommand{\tv}{\mathrm{TV}}

\newcommand{\Of}{\Omega_f}
\newcommand{\Oc}{\Omega_c}

\numberwithin{equation}{section}

\usepackage{calc}
\usepackage{authblk}
\author[1]{Edda Dal Santo}
\author[1]{Massimiliano D.\ Rosini}
\author[1]{Nikodem Dymski}
\author[2]{Mohamed Benyahia}

\affil[1]{Instytut Matematyki, Uniwersytet Marii Curie-Sk\l odowskiej,
   \centerline{pl.\ Marii Curie-Sk\l odowskiej~1, 20-031 Lublin, Poland}
   \centerline{dalsantoedda@gmail.com, mrosini@umcs.lublin.pl, ndymski@o2.pl}}
\affil[2]{Gran Sasso Science Institute
   \centerline{Viale F.\ Crispi 7, 67100 L'Aquila, Italy}
   \centerline{benyahia.ramiz@gmail.com}}
\date{}                     
\allowdisplaybreaks

\begin{document}

\setlength{\abovedisplayskip}{1pt}
\setlength{\belowdisplayskip}{1pt}
\setlength{\abovedisplayshortskip}{1pt}
\setlength{\belowdisplayshortskip}{1pt}

\title{General phase transition models for vehicular traffic\\with point constraints on the flow}

\maketitle

\begin{abstract}
We generalize the phase transition model studied in \cite{Colombo}, that describes the evolution of vehicular traffic along a one-lane road. 
Two different phases are taken into account, according to whether the traffic is low or heavy. 
The model is given by a scalar conservation law in the \emph{free-flow} phase and by a system of two conservation laws in the \emph{congested} phase.  
In particular, we study the resulting Riemann problems in the case a local point constraint on the flux of the solutions is enforced. 
\end{abstract}


\section{Introduction} %


This paper deals with phase transition models (PT models for short) of hyperbolic conservation laws for traffic. 
More precisely, we focus on models that describe vehicular traffic along a unidirectional one-lane road, which has neither entrances nor exits and where overtaking is not allowed.

In the specialized literature, vehicular traffic is shown to behave differently depending on whether it is free or congested. 
This leads to consider two different regimes corresponding to a \emph{free-flow phase} denoted by $\Of$ and a \emph{congested phase} denoted by $\Oc$. 
The PT models analyzed here are given by a scalar conservation law in the free-flow phase, coupled with a $2\times 2$ system of conservation laws in the congested phase. The coupling is achieved via \emph{phase transitions}, namely discontinuities that separate two states belonging to different phases and that satisfy the Rankine-Hugoniot conditions.

This two-phase approach was introduced by Colombo in \cite{Colombo} and is motivated by experimental observations, according to which for low densities the flow of vehicles is free and approximable by a one-dimensional flux function, while at high densities the flow is congested and covers a $2$-dimensional domain in the fundamental diagram, see \cite[Figure 1.1]{Colombo}. 
Hence, it is reasonable to describe the dynamics in the free regime by a first order model and those in the congested regime by a second order model. 

Colombo proposed to let the free-flow phase $\Of$ be governed by the classical LWR model by Lighthill, Whitham and  Richards  \cite{LighthillWhitham, Richards}, which expresses the conservation of the number of vehicles and assumes that the velocity is a function of the density alone; on the other hand, the congested phase $\Oc$ includes one more equation for the conservation of a linearized momentum.  
Furthermore, his model uses a Greenshields (strictly parabolic) flux function in the free-flow regime and one consequence is that $\Of$ cannot intersect $\Oc$, see \cite[Remark 2]{Colombo}. 

The two-phase approach was then exploited by other authors in subsequent papers, see \cite{BenyahiaRosini01, BenyahiaRosini02,Blandin,goatin2006aw}. For instance, in \cite{goatin2006aw} Goatin couples the LWR equation for the free-flow phase $\Of$ with the ARZ model formulated by Aw, Rascle and Zhang \cite{AwRascle, Zhang} for the congested phase $\Oc$. 
In the author's intentions such a model has the advantage of correcting the drawbacks of the LWR and ARZ models taken separately.
Recall that this PT model has been recently generalized in \cite{BenyahiaRosini01, BenyahiaRosini02}.
Another variant of the PT model of Colombo is obtained in \cite{Blandin}, where the authors take an arbitrary flux function in $\Oc$ and consider this phase as an extension of LWR that accounts for heterogeneous driving behaviours.

In this paper we further generalize the two PT models treated in  \cite{BenyahiaRosini01,BenyahiaRosini02,goatin2006aw} and \cite{Blandin}.
For more clarity, we refer to the first model as the PT$^p$ model and to the latter as the PT$^a$ model.
We omit these superscripts only when they are not necessary. 
We point out that in \cite{BenyahiaRosini01,BenyahiaRosini02,goatin2006aw} the authors assume that $\Of \cap \Oc = \emptyset$, while in \cite{Blandin} the authors assume that $\Of \cap \Oc \neq \emptyset$.
Here we do not impose any assumption on the intersection of the two phases for both the PT$^p$ and the PT$^a$ models. 
Moreover, in order to avoid the loss of well-posedness of the Riemann problems in the case $\Of \cap \Oc \neq \emptyset$ as noted in \cite[Remark 2]{Colombo}, we assume that the free phase $\Of$ is characterized by a unique value of the velocity, that coincides with the maximal one.

In the next sections, we study Riemann problems coupled with a local point constraint on the flow. More precisely, we analyze in detail two constrained Riemann solvers corresponding to the cases $\Of\cap \Oc\ne \emptyset$ and $\Of\cap \Oc=\emptyset$.
We recall that a local point constraint on the flow is a condition requiring that the flow of the solution at the interface $x=0$ does not exceed a given constant quantity $F$. 
This can model, for example, the presence at $x=0$ of a toll gate with capacity $F$.
We briefly summarize the literature on conservation laws with point constraint recalling that:
\begin{itemize}[leftmargin=*]\setlength{\itemsep}{0cm}%
\item
the LWR model with a local point constraint is studied analytically in \cite{colombogoatinconstraint, Rosini2013Chap6} and numerically in \cite{AndreianovGoatinSeguin, CancesSeguin, ChalonsGoatinSeguin, ColomboGoatinRosiniESAIM2011};
\item
the LWR model with a non-local point constraint is studied analytically in \cite{AndreianovDonadelloRazafisonRosiniMBE2015, AndreianovDonadelloRosiniM3ASS2014} and numerically in \cite{AndreianovDonadelloRazafisonRosiniESAIM2016};
\item
the ARZ model with a local point constraint is studied analytically in \cite{AndreianovDonadelloRosiniM3ASS2016, GaravelloGoatin2011, garavello2016cauchy} and numerically in \cite{AndreianovDonadelloRazafisonRollandRosiniNHM2016};
\item
the PT model \cite{BenyahiaRosini01,goatin2006aw} with a local point constraint is analytically studied in \cite{BenyahiaRosini02}.
\end{itemize}
To the best of our knowledge, the model presented in \cite{BenyahiaRosini02} is so far the only PT model with point constraint.

The paper is organized as follows. 
In Section~\ref{sec:PTmodels} we introduce the PT$^p$ and PT$^a$ models by giving a unified description valid in both cases. 
In particular, we list the basic notations and the main assumptions needed throughout the paper and we give a general definition of admissible solutions to a Riemann problem for a PT model.  
In sections~\ref{sec:inter} and \ref{sec:noninter} we outline the costrained Riemann solver in the case of intersecting phases and non-intersecting phases, respectively. 
Section~\ref{sec:tv} contains some total variation estimates that may be useful to compare the difficulty of applying the two solvers in a wave-front tracking scheme; see \cite{HoldenRisebroWFT} and the references therein.
In Section~\ref{sec:simu01} we apply the PT$^0$ model to compute an explicit example reproducing the effects of a toll gate on the traffic along a one-lane road.
Finally, in Section~\ref{sec:tech} we collect all the proofs of the properties previously stated. 


\section{The general PT models}\label{sec:PTmodels} %


In this section, we introduce the PT models, collect some useful notations and recall the main assumptions on the parameters already discussed in \cite{BenyahiaRosini01,Blandin}. 
See \figurename~\ref{fig:notations} for a picture with the main notations used throughout the article.

\begin{figure}[ht]
\centering{
\begin{psfrags}\scriptsize
      \psfrag{a}[l,t]{}
      \psfrag{b}[l,B]{}
      \psfrag{m}[l,B]{}
      \psfrag{c}[c,B]{$u_+^{f,c}\hphantom{a}$}
      \psfrag{d}[c,B]{$u_-^{f,c}\hphantom{a}$}
      \psfrag{l}[c,b]{$u_1^{f,c}\hphantom{a}$}
      \psfrag{e}[B,c]{}
      \psfrag{h}[B,l]{$V_{f,c}$}
      \psfrag{f}[B,l]{$f$}
      \psfrag{g}[c,c]{$\hphantom{aa}u_+$}
      \psfrag{i}[B,c]{$v(u_+)$}
      \psfrag{n}[c,B]{$\hphantom{a}u_*$}
      \psfrag{o}[l,c]{$\,u_2^+$}
      \psfrag{p}[c,c]{$u_2^-~$}
      \psfrag{q}[c,c]{$u_-$}
      \psfrag{R}[c,B]{$R$}
      \psfrag{s}[c,B]{$\rho_1^0(w_-)$}
      \psfrag{r}[l,B]{$\rho$}
\includegraphics[width=.3\textwidth]{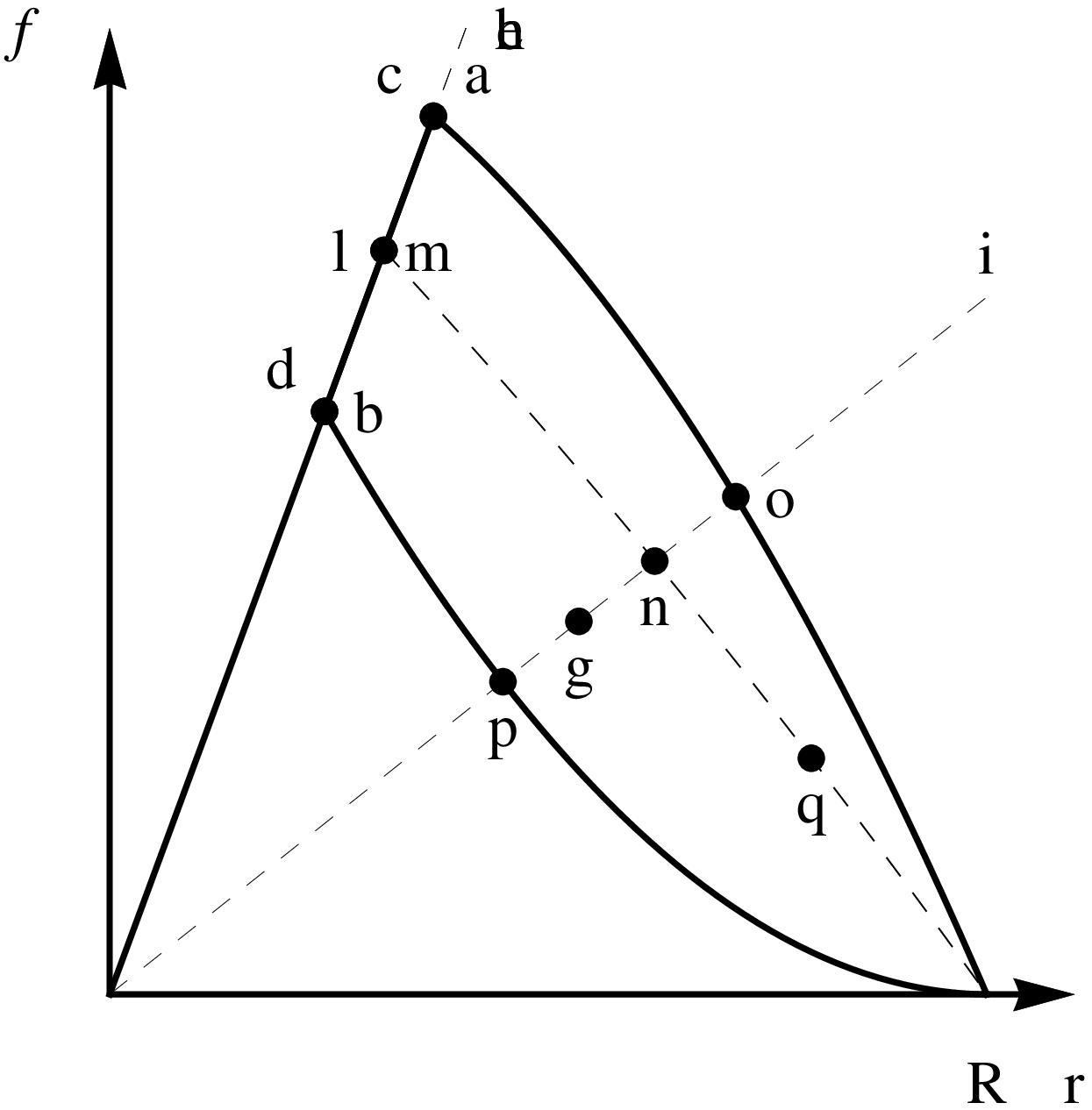}\qquad
      \psfrag{q}[l,t]{$u_-$}
      \psfrag{n}[l,b]{$\hphantom{a}u_*$}
\includegraphics[width=.3\textwidth]{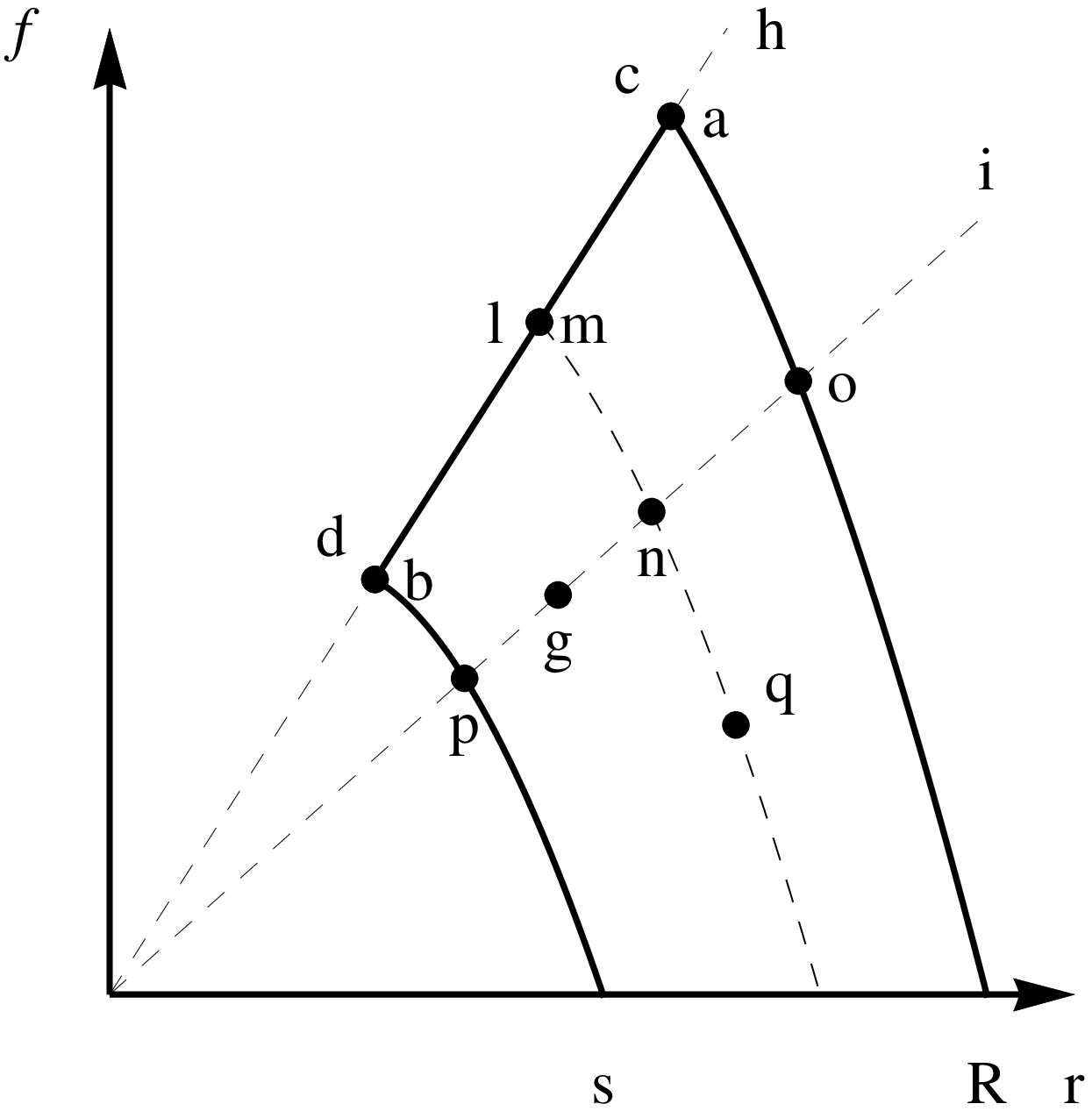}\\
      \psfrag{e}[B,c]{$\hphantom{a}V_c$}
      \psfrag{h}[B,c]{$\hphantom{a}V_f$}
      \psfrag{c}[c,B]{$u_+^f\hphantom{a}$}
      \psfrag{d}[c,B]{$u_-^f\hphantom{a}$}
      \psfrag{l}[c,b]{$u_1^f\hphantom{a}$}
      \psfrag{a}[l,t]{$\,u_+^c$}
      \psfrag{b}[l,B]{$\,u_-^c$}
      \psfrag{m}[l,B]{$\,u_1^c$}
      \psfrag{n}[c,B]{$\hphantom{a}u_*$}
      \psfrag{q}[c,c]{$u_-$}
\includegraphics[width=.3\textwidth]{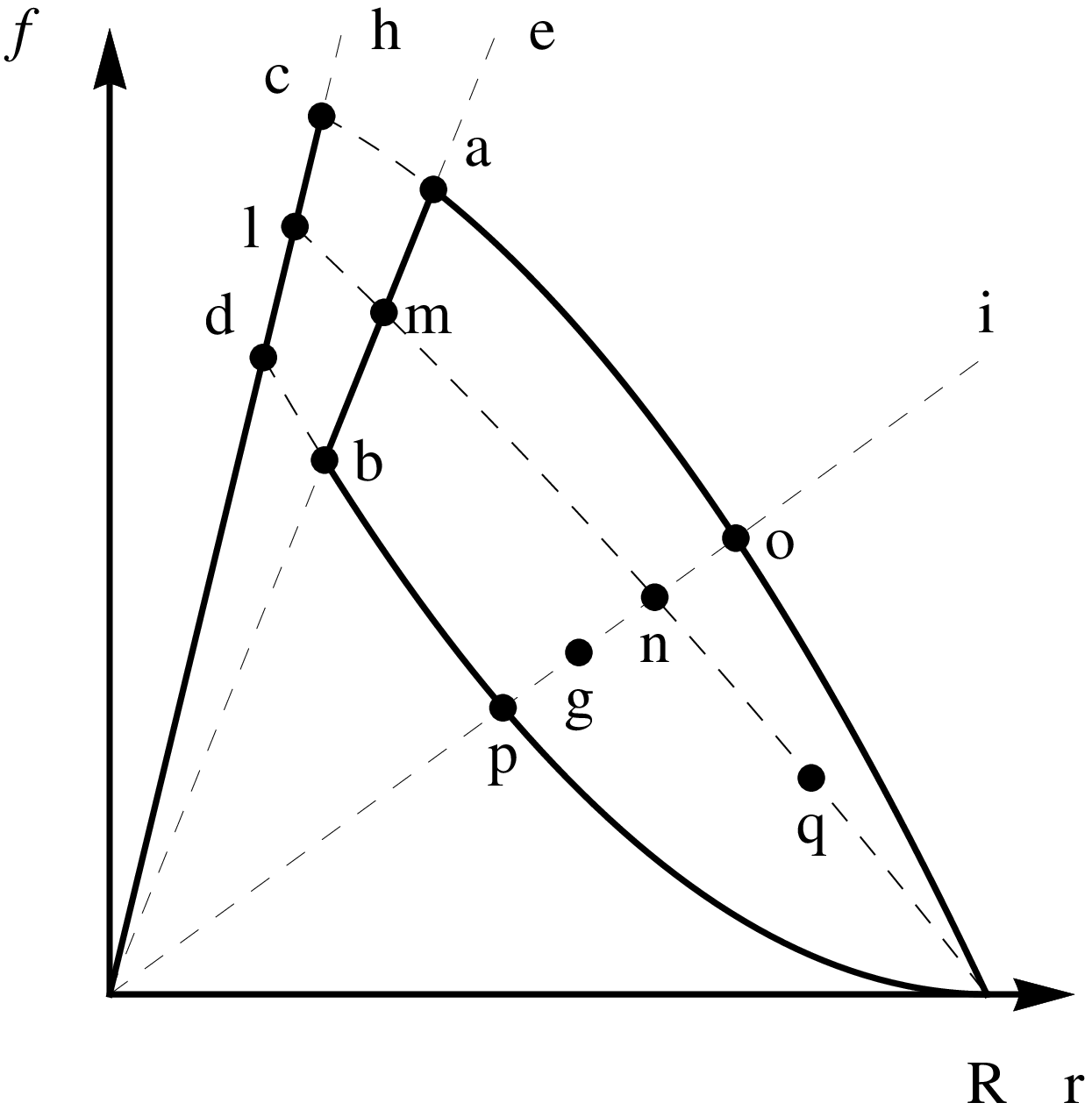}\qquad
      \psfrag{a}[l,B]{$\,u_+^c$}
      \psfrag{q}[l,t]{$u_-$}
      \psfrag{n}[l,b]{$\hphantom{a}u_*$}
\includegraphics[width=.3\textwidth]{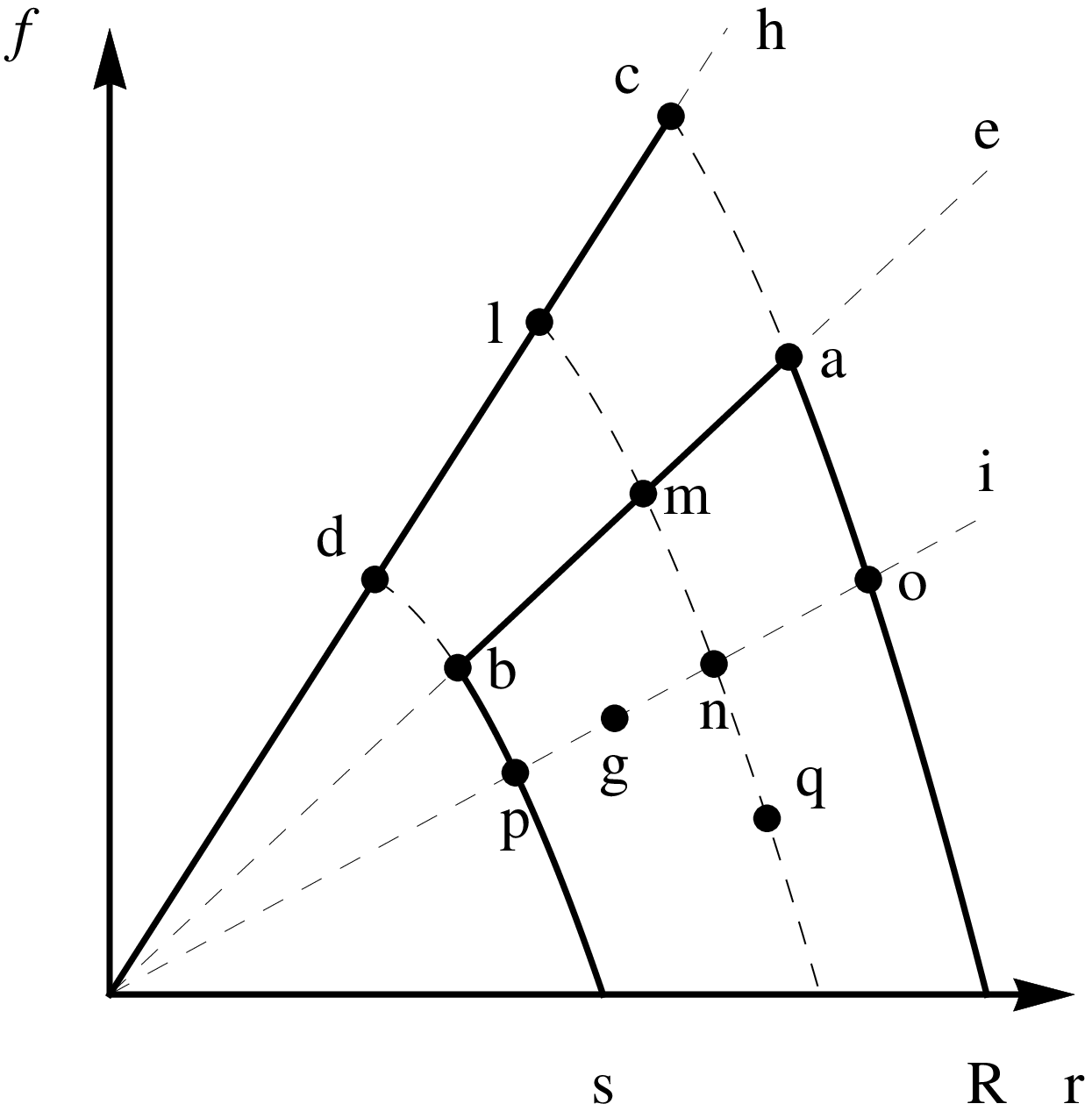}
\end{psfrags}}
\caption{Above we represent in the $(\rho,f)$-plane $u_* \doteq u_*(u_-,u_+)$, $u_\pm^f \doteq (\sigma_\pm^f,w_\pm \, \sigma_\pm^f)$, $u_\pm^c \doteq (\sigma_\pm^c,w_\pm \, \sigma_\pm^c)$, $u_2^\pm \doteq \psi_2^\pm(u_+)$, $u_1^f \doteq \psi_1^f(u_-)$ and $u_1^c \doteq \psi_1^c(u_-)$ defined in Section~\ref{sec:PTmodels}.
The pictures on the left refer to the PT$^a$ model, those on the right to the PT$^p$ model; the pictures in the first line refer to the case $\Of\cap\Oc \ne \emptyset$, namely $V_c = V_f$, and those in the last line to the case $\Of\cap\Oc = \emptyset$, namely $V_c < V_f$.}
\label{fig:notations}
\end{figure}
\noindent
The fundamental parameters that are common to PT$^a$ and PT$^p$ are the following:
\begin{itemize}
\item $V_f > 0$ is the unique velocity in the free phase, namely it is the maximal velocity;
\item $V_c \in (0,V_f]$ is the maximal velocity in the congested phase;
\item $R>0$ is the maximal density (in the congested phase).
\end{itemize}

We rewrite both models as:
\begin{align}\label{eq:system}
&\begin{array}{l}
\text{\textbf{Free flow}}\\[2pt]
\begin{cases}
u \doteq (\rho,q) \in \Of,\\
\rho_t+f(u)_x=0,\\
v(u)=V_f,
\end{cases}
\end{array}
&
\begin{array}{l}
\text{\textbf{Congested flow}}\\[2pt]
\begin{cases}
u \doteq (\rho,q) \in \Oc,\\
\rho_t+f(u)_x=0,\\
q_t + \left[q\,v(u)\right]_x=0.
\end{cases}
\end{array}
\end{align}
Above, $\rho \in [0,R]$ represents the density and $q$ the (linearized) momentum of the vehicles, while $\Of$ and $\Oc$ denote the domains of the free-flow  phase and of the congested phase, respectively. Observe that in $\Of$ the density $\rho$ is the unique independent variable, while in $\Oc$ the independent variables are both $\rho$ and $q$.
Moreover, the (average) speed $v \ge 0$ and the flow $f \ge 0$ of the vehicles are defined as
\begin{align*}
&v(u)\doteq \begin{cases}
v^a(u) \doteq v_{eq}^a(\rho) \, (1+q) &\text{for the PT$^a$ model}, \\[2pt]
v^p(u) \doteq\frac{q}{\rho} - p(\rho) &\text{for the PT$^p$ model},
\end{cases}
&f(u) \doteq \rho \, v(u).
\end{align*}
In the PT$^a$ model, $v_{eq}^a(\rho)$ is the equilibrium velocity defined by
\[
(0,R] \ni \rho \mapsto
v_{eq}^a(\rho) \doteq \left(\dfrac{R}{\rho} - 1\right)\left(\dfrac{V_f \,\sigma}{R-\sigma} + a \, (\sigma - \rho)\right),
\]
where $a \in \R$ and $\sigma \in (0,R)$ are fixed parameters, while the term $(1+q)>0$ is a perturbation which provides a thick fundamental diagram in the congested phase (in accordance with the experimental observations depicted in \cite[Figure 3.1]{Blandin}).
\begin{remark}
Observe that for any $\alpha < 0$ and $V \ge V_f$ we have that $v_{eq}^a$ coincides with the Newell-Daganzo \cite{Daganzo, Newell} velocity $\rho \mapsto \min\{V,\alpha \, (1-R/\rho)\}$ on $[\sigma,R]$ for $a=0$ and $\sigma = R \, (1+V_f/\alpha)^{-1}$.
Moreover, for any $V > V_f$ we have that $v_{eq}^a$ reduces to the Greenshields \cite{Greenshields} velocity $\rho \mapsto V(1-\rho/R)$ for $a=-V/R$ and $\sigma = R \, (1-V_f/V)$.
Observe that $v_{eq}^a(\sigma) = V_f$ by definition.
\end{remark}
\noindent
On the other hand, in the PT$^p$ model we require that $p \colon (0,R] \to \R$ satisfies
\begin{align}\tag{P}\label{P}
    &p \in \C2((0,R];\R),&
    &p'(\rho)>0,&
    &2p'(\rho) + \rho \, p''(\rho)>0&
    \text{for every }\rho \in (0,R].
\end{align}
A typical choice is $p(\rho) \doteq \rho^\gamma$, $\gamma>0$, see \cite{AwRascle}.

Let $\sigma_\pm^f$, $\sigma_\pm^c$ and $q_\pm$ be fixed parameters such that
\begin{align*}
&q_- < q_+,&
&0<\sigma_-^f < \sigma_+^f < R,&
&v(\sigma_\pm^f,\sigma_\pm^f \, q_\pm/R)=V_f,
\\&&
&0<\sigma_-^c < \sigma_+^c<R,&
&v(\sigma_\pm^c,\sigma_\pm^c \, q_\pm/R)=V_c.
\end{align*}
By definition $\sigma_\pm^f \le \sigma_\pm^c$, with the equality holding if and only if $V_f = V_c$.
Then, we can introduce the free and congested domains
\begin{align*}
&\Of \doteq \Bigl\{ u \in [0,\sigma_+^f] \times \R \,:\, q = Q(\rho)\Bigr\},&
&\Oc \doteq 
\Bigl\{ u \in [\sigma_-^c,R] \times \R \,:\, 0\le v(u)\le V_c,\, w_- \le \dfrac{q}{\rho} \le w_+ \Bigr\},
\end{align*}
where $w_\pm \doteq q_\pm/R $ and
\[
Q(\rho)\doteq
\begin{cases}
\dfrac{(\rho-\sigma)[V_f R+a(R-\rho)(R-\sigma)]}{(R-\rho)[V_f\sigma +a \, (\sigma-\rho)(R-\sigma)]} &\text{for the PT$^a$ model},\\[10pt]
\rho \left[V_f+p(\rho)\right]	&\text{for the PT$^p$ model}.
\end{cases}
\]
Observe that $v(u) = V_f$ for any $u \in \Of$ and $w_\pm= Q(\sigma_\pm^c)/\sigma_\pm^c=Q(\sigma_\pm^f)/\sigma_\pm^f$.
Moreover,
\[
u_-^c \doteq (\sigma_-^c, w_- \, \sigma_-^c)
\]
is the point in $\Oc$ with minimal $\rho$-coordinate.
Furthermore, we denote $\Omega \doteq \Of \cup \Oc$ and
\begin{align*}
&\Of^-\doteq \bigl\{u\in \Of \,:\, \rho \in [0,\sigma_-^f) \bigr\},&
&\Of^+\doteq \bigl\{u\in \Of \,:\, \rho \in [\sigma_-^f,\sigma_+^f] \bigr\},
\\
&\Oc^- \doteq \Oc\setminus\Of^+,&
&\Oc^{ex} \doteq \bigl\{ u \in (0,R] \times \R \,:\, v(u) \in [0, V_f] ,\ w(u) \in [w_-,w_+]\bigr\},
\end{align*}
where
\[
w(u) \doteq
\begin{cases}
q/\rho&\text{if }u \in \Oc^{ex},
\\
w_-+V_f \left(\dfrac{\rho}{\sigma_-^f}-1\right) &\text{if }u \in \Of^- .
\end{cases}
\]
We point out that
\begin{align*}
&V_c = V_f&&\Rightarrow&
&\Of\cap\Oc = \Of^+,&
&\Oc^- \subset \Oc,&
&\Oc^{ex} = \Oc,
\\
&V_c < V_f&&\Rightarrow&
&\Of\cap\Oc = \emptyset,&
&\Oc^- = \Oc,&
&\Oc^{ex} \supset \Oc\cup\Of^+.
\end{align*}

\begin{remark}
Observe that in the congested phase $w$ is a Lagrangian marker, since it satisfies $w(u)_t + v(u) \, w(u)_x = 0$ as long as the solution $u$ to \eqref{eq:system} attains values in $\Oc$.
\end{remark}

We introduce functions that are practical in the definition of the Riemann solvers given in the next sections:
\begin{align*}
&\psi_1^{f,c} \colon (\Oc\cup\Of^+) \to (\Oc\cup\Of^+),&
&u=\psi_1^{f,c}(u_o):
\begin{cases}
w(u)=w(u_o),\\
v(u)=V_{f,c},
\end{cases}
\\
&\rho_1^0 \colon [w_-,w_+] \to (0,R],&
&\rho_1^0(w) \doteq
\begin{cases}
R &\text{for the PT$^a$ model}, \\
p^{-1}(w) &\text{for the PT$^p$ model},
\end{cases}
\\
&\psi_2^\pm \colon \Omega \to (\Oc\cup\Of^+),&
&u=\psi_2^\pm(u_o):
\begin{cases}
w(u)=w_\pm,\\
v(u)=v(u_o),
\end{cases}
\\
&u_*\colon(\Oc\cup\Of^+)^2 \to (\Oc\cup\Of^+),&
&u=u_*(u_-,u_+) :\begin{cases}
w(u) = w(u_-),\\
v(u) = v(u_+),
\end{cases}
\\
&\Lambda \colon \left\{(u_-,u_+) \in \Omega^2 \,:\, \rho_- \neq \rho_+\right\} \to \R,&
&\Lambda(u_-, u_+) \doteq \frac{f(u_+)-f(u_-)}{\rho_+-\rho_-}.
\end{align*}
We underline that $\rho_1^0(w_+) = R$ both for the PT$^a$ and the PT$^p$ models.
Observe that, according to the definition of the Lax curves given in the next section, the above functions have the following geometrical meaning, see Figure~\ref{fig:notations}:
\begin{itemize}[leftmargin=*]\setlength{\itemsep}{0cm}%
\item
the point $\psi_1^{f,c}(u_o)$ is the intersection of the Lax curve of the first family passing through $u_o$ and $\{u \in \Omega \,:\, v(u) = V_{f,c}\}$;
\item
for any $w \in [w_-,w_+]$ the point $(\rho_1^0(w), \rho_1^0(w) \, w)$ is the intersection of the Lax curve of the first family corresponding to $w$ and $\{u \in \Oc \,:\, v(u) = 0\}$;
\item
the point $\psi_2^\pm(u_o)$ is the intersection of the Lax curve of the second family passing through $u_o$ and $\{u \in \Omega \,:\, w(u) = w_\pm\}$;
\item
the point $u_*(u_-,u_+)$ is the intersection between the Lax curve of the first family passing through $u_-$ and the Lax curve of the second family passing through $u_+$;
\item
$\Lambda(u_-, u_+)$ is the Rankine-Hugoniot speed connecting two states $u_-,u_+$, namely it is the slope of the segment connecting $u_-$ and $u_+$.
\end{itemize}
Finally, we introduce the maps
\begin{align*}
&\rho_1^{f,c} \colon [w_-,w_+] \to [\sigma_-^{f,c},\sigma_+^{f,c}]&
& \text{ and }&
&\rho_2^\pm \colon [0,V_f] \to [\sigma_\pm^f,\rho_1^0(w_\pm)]
\end{align*}
such that $\rho_1^{f,c}(w(u_o))$ and $\rho_2^\pm(v(u_o))$ are respectively the $\rho$-components of $\psi_1^{f,c}(u_o)$ and $\psi_2^\pm(u_o)$.


\subsection{Main assumptions} %


The model consists of a scalar conservation law in the free-flow regime and of a $2\times 2$ system of conservation laws in the congested one. In the free-flow phase the characteristic speed is $V_f$.
Below we collect the eigenvalues, eigenvectors and Riemann invariants for the system in the congested phase:
\begin{align*}
&\lambda_1(u) \doteq v(u) + u \cdot \nabla v(u),&
&\lambda_2(u) \doteq v(u), \\
&r_1(u) \doteq \begin{pmatrix}
\rho\\
q
\end{pmatrix},& 
&r_2(u) \doteq \begin{pmatrix}
 v_q(u)\\
- v_\rho(u)
\end{pmatrix},\\
&w_1(u) \doteq w(u),& 
&w_2(u) \doteq v(u).
\end{align*}
For later use, we consider the natural extension of the above functions to $\Oc^{ex}$.
We observe that $\lambda_2(u) \ge 0$ for all $u\in\Oc^{ex}$ with the equality holding if and only if $\rho = \rho_1^0(w(u))$.
For simplicity, we assume that
\begin{gather}\tag{H1}\label{H1}
\lambda_1(u) < 0 \text{ for all }u\in\Oc^{ex}.
\end{gather}
As a consequence, in the congested phase the waves of the first characteristic family have negative speed and those of the second family have non-negative speed.

Further computations show that
\begin{align*}
&\nabla\lambda_1(u)\cdot r_1(u)=
2 \, u \cdot \nabla v(u) + \rho^2 \, v_{\rho\rho}(u) +2 \, \rho \, q \, v_{\rho q}(u) + q^2 \, v_{qq}(u),
\\
&\nabla\lambda_2(u)\cdot r_2(u)= 0,
\end{align*}
and we can infer that the second characteristic field is linearly degenerate.
Moreover, for simplicity we assume that
\begin{equation}\tag{H2}\label{H2}
\text{the first characteristic field is genuinely nonlinear in $\Oc^{ex}$, except for the PT$^a$ model with }a = 0.
\end{equation}
We point out that for the PT$^0$ model the first characteristic field is genuinely nonlinear except in $\{u \in \Oc^{ex} \,:\, w(u) = 0\}$.
\begin{remark}
Concerning the PT$^p$ model, \eqref{H1} corresponds to require 
\begin{align*}
 &V_f < \rho \, p'(\rho),& &\text{ for every }\rho \in [\sigma_-^f,\sigma_+^f].
\end{align*}
Let us underline that if we ask that $p$ satisfies also $p'(\rho)+\rho\, p''(\rho)>0$ for every $\rho\in[\sigma_-^f,\sigma_+^f]$ as done in \cite{AndreianovDonadelloRazafisonRollandRosiniNHM2016, AndreianovDonadelloRosiniM3ASS2016},
then the above condition reduces to $\sigma_-^f\,p'(\sigma_-^f)>V_f$.
Furthermore, by \eqref{P} we have $\nabla\lambda_1(u)\cdot r_1(u) = -\rho \, \bigl(2p'(\rho) + \rho \, p''(\rho)\bigr) < 0$ for every $\rho \in (0,R]$ and \eqref{H2} easily follows.

On the other hand, for the PT$^a$ model in general \eqref{H1} and \eqref{H2} cannot be easily expressed in terms of the parameters of the model.
Here we just recall from \cite{GaravelloPiccoli} that in the simplest case $a=0$ \eqref{H1} is guaranteed by
\[
-\frac{1}{R} < w_-<0<w_+<\frac{1}{R}.
\]
\end{remark}

In the $(\rho,f)$-plane the Lax curves of the first and second characteristic families passing through a fixed point $u_o\in \Oc^{ex}$ are described respectively by the graphs of the maps
\begin{align*}
&[\rho_1^f(w(u_o)),\rho_1^0(w(u_o))] \ni \rho \mapsto L_{w(u_o)}(\rho) \doteq f(\rho, w(u_o) \, \rho),
&[\rho_2^-(v(u_o)),\rho_2^+(v(u_o))]\ni \rho \mapsto v(u_o) \, \rho.
\end{align*}

\begin{remark}\label{rem:Mars}
We point out that \eqref{H1} and \eqref{H2} can be reformulated in terms of the first Lax curves.
Indeed, since $L_{w}'(\rho)  = \lambda_1(\rho,w\,\rho)$, we have that \eqref{H1} is equivalent to require that the first Lax curves are strictly decreasing, so that the \emph{capacity drop} in the passage from the free phase to the congested phase is ensured.
On the other hand, since $L_{w}''(\rho) = \frac{1}{\rho} \, \nabla\lambda_1(\rho,w\,\rho)\cdot r_1(\rho,w\,\rho)$, we also have that \eqref{H2} is equivalent to require that the first Lax curves are \emph{strictly} concave or convex, except for the one for the PT$^0$ model corresponding to $w=0$.
In particular, since by \eqref{P} $L_w''(\rho) < 0$ for all $\rho\in[\rho_1^f(w),\rho_1^0(w)]$, for the PT$^p$ model we have that the first Lax curves are strictly concave.
\end{remark}


\subsection{The constrained Riemann problem} %


Let us consider the Riemann problem for the PT model, namely the Cauchy problem for \eqref{eq:system} with initial datum
\begin{equation}\label{eq:Rdata}
u(0,x)=\begin{cases}
u_\ell &\hbox{if }x<0,\\ 
u_r &\hbox{if }x>0.
\end{cases}
\end{equation}
We recall the following general definition of solution to \eqref{eq:system},\eqref{eq:Rdata} given in \cite[p.\ 712]{Colombo}.

\begin{definition}\label{def:Colombo}
For any $u_\ell, u_r \in \Omega$, an \emph{admissible} solution to the Riemann problem \eqref{eq:system},\eqref{eq:Rdata} is a self-similar function $u \doteq (\rho, q) \colon \R_+ \times \R \to \Omega$ that satisfies the following conditions.

\begin{itemize}

\item If $u_\ell,u_r\in \Of$ or $u_\ell,u_r\in \Oc$, then $u$ is the usual Lax solution to \eqref{eq:system},\eqref{eq:Rdata} (and it does not perform any phase transition).

\item If $u_\ell\in\Of^-$ and $u_r\in\Oc^-$, then there exists $\Lambda\in \R$ such that:

\begin{itemize}

\item
$u(t,(-\infty,\Lambda\,t)) \subseteq \Of$ and $u(t,(\Lambda\,t,+\infty)) \subseteq \Oc$ for all $t >0$;

\item
the Rankine-Hugoniot jump conditions
\[
\Lambda \, [\rho(t,\Lambda \, t^+) - \rho(t,\Lambda \, t^-) ] = f(u(t,\Lambda \, t^+)) - f(u(t,\Lambda \, t^-))
\]
are satisfied for all $t >0$;

\item 
the functions
\begin{align*}
(t,x) \mapsto \begin{cases}
u(t,x)&\text{if } x<\Lambda \, t,
\\
u(t,\Lambda \, t^-)&\text{if } x>\Lambda \, t,
\end{cases}
&&
(t,x) \mapsto \begin{cases}
u(t,\Lambda \, t^+)&\text{if } x<\Lambda \, t,
\\
u(t,x)&\text{if } x>\Lambda \, t,
\end{cases}
\end{align*}
are respectively the usual Lax solutions to the Riemann problems
\begin{align*}
\begin{cases}
\rho_t + f(\rho)_x=0,\\
v(u) = V_f,\\
u(0,x)=\begin{cases}
u_\ell &\hbox{if }x<0,\\ 
u(t,\Lambda \, t^-) &\hbox{if }x>0,
\end{cases}
\end{cases}
&
\begin{cases}
\rho_t+f(u)_x=0,\\
q_t + \left[q\,v(u)\right]_x=0,\\
u(0,x)=\begin{cases}
u(t,\Lambda \, t^+) &\hbox{if }x<0,\\ 
u_r &\hbox{if }x>0.
\end{cases}
\end{cases}
\end{align*}

\end{itemize}

\item
If $u_\ell\in\Oc^-$ and $u_r\in \Of^-$, conditions entirely analogous to the previous case are required.
\end{itemize}
\end{definition}

We denote by $\mathcal{R}$ and $\mathcal{S}$ the Riemann solvers associated to the Riemann problem \eqref{eq:system},\eqref{eq:Rdata}, respectively in the cases of intersecting and non-intersecting phases.
We point out that these Riemann solvers are defined below according to Definition~\ref{def:Colombo}, in the sense that $(t,x) \mapsto \mathcal{R}[u_\ell,u_r](x/t)$ and $(t,x) \mapsto \mathcal{S}[u_\ell,u_r](x/t)$ are admissible solutions to the Riemann problem \eqref{eq:system},\eqref{eq:Rdata}.

Besides the initial condition \eqref{eq:Rdata}, we enforce a local point constraint on the flow at $x=0$, i.e.\ we add the further condition that the flow of the solution at the interface $x=0$ is lower than a given constant quantity $F \in (0,V_f\,\sigma_+^f)$ and impose
\begin{equation}\label{eq:const}
f(u(t,0^\pm)) \le F.
\end{equation}
In general, \eqref{eq:const} is not satisfied by an admissible solution to \eqref{eq:system},\eqref{eq:Rdata}.
For this reason we introduce the following concept of admissible constrained solution to \eqref{eq:system},\eqref{eq:Rdata},\eqref{eq:const}.  

\begin{definition}\label{def:42}
For any $u_\ell,u_r\in\Omega$, an \emph{admissible constrained} solution to the  Riemann problem \eqref{eq:system}, \eqref{eq:Rdata},\eqref{eq:const} is a self-similar function $u \doteq (\rho, q)\colon \R_+\times\R\to \Omega $ such that $\hat{u}\doteq u(t,0^-)$ and $\check{u}\doteq u(t,0^+)$ satisfy:

\begin{itemize}

\item $f(\hat{u})=f(\check{u})\le F$;

\item the functions
\begin{align*}
(t,x) \mapsto \begin{cases}
u(t,x)&\text{if } x<0,
\\
\hat{u}  &\text{if } x>0,
\end{cases}
&&
(t,x) \mapsto \begin{cases}
\check{u} &\text{if } x<0,
\\
u(t,x)&\text{if } x>0,
\end{cases}
\end{align*}
are admissible solutions to the Riemann problems for \eqref{eq:system} with  Riemann data respectively given by
\begin{align*}
&u(0,x) = \begin{cases}
u_\ell &\text{if } x<0,
\\
\hat{u}&\text{if } x>0,
\end{cases}
&
u(0,x) = \begin{cases}
\check{u} &\text{if } x<0,
\\
u_r&\text{if } x>0.
\end{cases}
\end{align*}

\end{itemize}
\end{definition}

We denote by $\mathcal{R}_F$ and $\mathcal{S}_F$ the constrained Riemann solvers associated to the Riemann problems \eqref{eq:system},\eqref{eq:Rdata},\eqref{eq:const}, respectively in the cases of intersecting and non-intersecting phases.
We point out that these Riemann solvers are defined below according to Definition~\ref{def:42}, in the sense that $(t,x) \mapsto \mathcal{R}_F[u_\ell,u_r](x/t)$ and $(t,x) \mapsto \mathcal{S}_F[u_\ell,u_r](x/t)$ are admissible constrained solutions.

We let (with a slight abuse of notation)
\begin{align*}
&\mathcal{R}_F \doteq \mathcal{R} &\text{ in }&& \mathcal{D}_1 \doteq \{ (u_\ell,u_r) \in \Omega^2 \,:\, f(\mathcal{R}[u_\ell,u_r](t,0^\pm)) \le F \},
\\
&\mathcal{S}_F \doteq \mathcal{S} &\text{ in }&& \mathcal{D}_1 \doteq \{ (u_\ell,u_r) \in \Omega^2 \,:\, f(\mathcal{S}[u_\ell,u_r](t,0^\pm)) \le F \},
\end{align*}
and we denote $\mathcal{D}_2 \doteq \Omega^2 \setminus \mathcal{D}_1$.

In the next sections, we introduce the Riemann solvers and discuss their main properties, such as their consistency, $\Lloc1$-continuity and their invariant domains. 
In this regard, we recall the following definitions. 

\begin{definition}\label{def:cons}
A Riemann solver $\mathcal{T} \colon \Omega^2 \to \L\infty(\R;\Omega)$ is said to be consistent if it satisfies both the following statements for any $u_\ell, u_m, u_r \in \Omega$ and $\bar x \in \R$:
\begin{align}\label{P1}\tag{I}
  \mathcal{T}[u_\ell,u_r](\bar x) = u_m&
  & \Rightarrow &
  &&\begin{cases}
     \mathcal{T}[u_\ell,u_m](x)= 
      \begin{cases}
         \mathcal{T}[u_\ell,u_r](x)& \hbox{if } x < \bar x ,
          \\
          u_m & \hbox{if } x \ge \bar x ,
      \end{cases}
      \\[10pt]
      \mathcal{T}[u_m,u_r](x) =
      \begin{cases}
          u_m & \hbox{if } x < \bar x ,
          \\
         \mathcal{T}[u_\ell,u_r](x)& \hbox{if } x \geq \bar x,
      \end{cases}
      \end{cases}
      \\
      \begin{rcases}
 \mathcal{T}[u_\ell,u_m](\bar x)=u_m 
  \\\label{P2}\tag{II}
  \mathcal{T}[u_m,u_r](\bar x)=u_m
  \end{rcases}&
  & \Rightarrow &
  &&\mathcal{T}[u_\ell,u_r](x)=
      \begin{cases}
     \mathcal{T}[u_\ell,u_m](x)& \hbox{if } x < \bar x ,
      \\
     \mathcal{T}
     [u_m,u_r](x) & \hbox{if } x \geq \bar x .
      \end{cases}
\end{align}
\end{definition}
We point out that the consistency of a Riemann solver is a necessary condition for the well-posedness of the Cauchy problem in $\bf{L^1}$.

\begin{definition}
An invariant domain for $\mathcal{T}$ is a set $\mathcal{I}\subseteq \Omega$ such that $\mathcal{T}[\mathcal{I},\mathcal{I}](\R) \subseteq \mathcal{I}$.  
\end{definition}


\section{The PT models with intersecting phases}\label{sec:inter}


In this section we consider the case in which $\Of \cap \Oc = \Of^+ \ne \emptyset$, namely the maximal velocities for the free phase and congested phase coincide, $V_f = V_c$.
For notational simplicity, we call 
\begin{align*}
&V \doteq V_f = V_c,
&\psi_1 \doteq \psi_1^{f}=\psi_1^{c},&
&\sigma_- \doteq \sigma_-^{f} = \sigma_-^{c}.
\end{align*}
Below we give the definitions of the Riemann solver $\mathcal{R}$ and of the constrained Riemann solver $\mathcal{R}_F$, which are valid for both the PT$^a$ and PT$^p$ models. 

\begin{definition}\label{def:R}
The Riemann solver $\mathcal{R} \colon \Omega^2 \to \L\infty(\R;\Omega)$  associated to the Riemann problem \eqref{eq:system},\eqref{eq:Rdata} is defined as follows.

\begin{enumerate}[label={(R.\arabic*)},leftmargin=*]\setlength{\itemsep}{0cm}%

\item If $u_\ell,u_r\in \Of$, then 
the solution consists of a contact discontinuity from $u_\ell$ to $u_r$ with speed $V$.

\item If $u_\ell,u_r\in\Oc$, then the solution consists of a $1$-wave from $u_\ell$ to $u_*(u_\ell,u_r)$ and of a $2$-contact discontinuity from $u_*(u_\ell,u_r)$ to $u_r$.

\item If $u_\ell\in\Oc^-$ and $u_r\in\Of^-$, then the solution consists of a $1$-wave from $u_\ell$ to $\psi_1(u_\ell)$ and a contact discontinuity from $\psi_1(u_\ell)$ to $u_r$.

\item \label{R4} If $u_\ell\in\Of^-$, $u_r\in \Oc^-$ and $\Lambda(u_\ell,\psi_2^-(u_r))\ge \lambda_1(\psi_2^-(u_r))$, then the solution consists of a phase transition from $u_\ell$ to $\psi_2^-(u_r)$ and a $2$-contact discontinuity from $\psi_2^-(u_r)$  to $u_r$.

\item \label{R4b} If $u_\ell\in\Of^-$, $u_r\in \Oc^-$ and  $\Lambda(u_\ell,\psi_2^-(u_r))< \lambda_1(\psi_2^-(u_r))$, then let $u_p = u_p(u_\ell)$ be the state satisfying $w(u_p)=w_-$ and $\Lambda(u_\ell,u_p
)=\lambda_1(u_p)$. In this case, the solution consists of a phase transition from $u_\ell$ to $u_p$, a $1$-rarefaction from $u_p$ to $\psi_2^-(u_r)$ and a $2$-contact discontinuity from $\psi_2^-(u_r)$ to $u_r$.

\end{enumerate}
\end{definition}

\noindent Notice that $L_{w_-}''(\sigma_-)\le 0$ implies that $\Lambda(u_\ell,\psi_2^-(u_r)) \ge \lambda_1(\psi_2^-(u_r))$ for all $u_\ell\in\Of^-$, $u_r\in \Oc^-$ and, hence, \ref{R4b} never occurs.
In particular, by \eqref{P} this is the case for the PT$^p$ model, see Remark~\ref{rem:Mars}.

The next proposition lists the main properties of $\mathcal{R}$. For the proof we defer to Section~\ref{sec:tec0}.

\begin{proposition}\label{prop:cc0}
The Riemann solver $\mathcal{R}$ is $\Lloc1$-continuous and consistent. 
\end{proposition}

Before introducing the Riemann solver $\mathcal{R}_F$, we observe that in the present case
\[\begin{array}{r@{}c@{\,}l}
\mathcal{D}_1 = &&
\{ (u_\ell , u_r) \in \Of^2 \,:\, f(u_\ell) \le F \}
\cup
\{ (u_\ell , u_r) \in \Oc^2 \,:\, f(u_*(u_\ell,u_r)) \le F \}
\\[2pt]&\cup&
\{ (u_\ell , u_r) \in \Oc^- \times \Of^- \,:\, f(\psi_1(u_\ell)) \le F \}
\\[2pt]&\cup&
\{ (u_\ell , u_r) \in \Of^- \times \Oc^- \,:\, \min\{f(u_\ell),f(\psi_2^-(u_r))\} \le F \}.
\end{array}\]

\begin{definition}\label{def:01}
The constrained Riemann solver $\mathcal{R}_F \colon \Omega^2 \to \L\infty(\R;\Omega)$ associated to \eqref{eq:system},\eqref{eq:Rdata},\eqref{eq:const} is defined as
\[
\mathcal{R}_F[u_\ell,u_r](x)\doteq \begin{cases}
\mathcal{R}[u_\ell,u_r](x) &\hbox{if }(u_\ell,u_r) \in \mathcal{D}_1,\\[5pt]
\begin{cases}
\mathcal{R}[u_\ell,\hat{u}](x) &\hbox{if }x<0,\\
\mathcal{R}[\check{u},u_r](x) &\hbox{if }x>0,
\end{cases}
&\hbox{if }(u_\ell,u_r) \in \mathcal{D}_2,
\end{cases}
\]
where $\hat{u}  = \hat{u}(u_\ell,F) \in \Oc$ and $\check{u}=\check{u}(u_r,F)  \in \Omega$ are uniquely selected by the conditions
\begin{align*}
&f(\hat{u})=f(\check{u})=F,&
&w(\hat{u}) = \max\left\{w(u_\ell),w_-\right\},&
&v(\check{u})=\begin{cases}
V&\hbox{if }f(\psi_2^-(u_r))> F,\\
v_r &\hbox{if }f(\psi_2^-(u_r))\le F.
\end{cases}
\end{align*}
\end{definition}
\noindent
In \figurename~\ref{fig:uhatucheck0} we specify the selection criterion for $\hat{u}$ and $\check{u}$ given above in all the possible cases. 
We point out that $\hat{u}$ and $\check{u}$ satisfy the following general properties.
\begin{equation}\label{gen.properties}
\begin{minipage}{.64\textwidth}
If $(u_\ell,u_r) \in \mathcal{D}_2$, then $w(u_\ell) > w(\check{u})$ and $v(u_r) > v(\hat{u})$.\\
If $(u_\ell,u_r) \in \mathcal{D}_2$ and $u_\ell \in \Of^-$, then $w(\hat{u}) = w_-$.\\
If $(u_\ell,u_r) \in \mathcal{D}_2$ and $u_r \in \Of$, then $v(\check{u}) = V$.
\end{minipage}
\end{equation}

\begin{figure}
    \centering
    \begin{subfigure}[b]{0.24\textwidth}
    \begin{psfrags}
      \psfrag{f}[c,c]{$f$}
      \psfrag{R}[c,B]{$R$}
      \psfrag{r}[c,B]{$\rho$}
      \psfrag{1}[c,B]{$F$}
      \psfrag{0}[l,B]{$u_r$}
      \psfrag{2}[l,B]{$\check{u}$}
      \psfrag{3}[c,B]{$\hat{u}_2$}
      \psfrag{4}[c,B]{$\hat{u}_1$}
      \psfrag{5}[c,B]{$u_\ell^2~$}
      \psfrag{6}[c,B]{$u_\ell^1~$}
      \includegraphics[width=\textwidth]{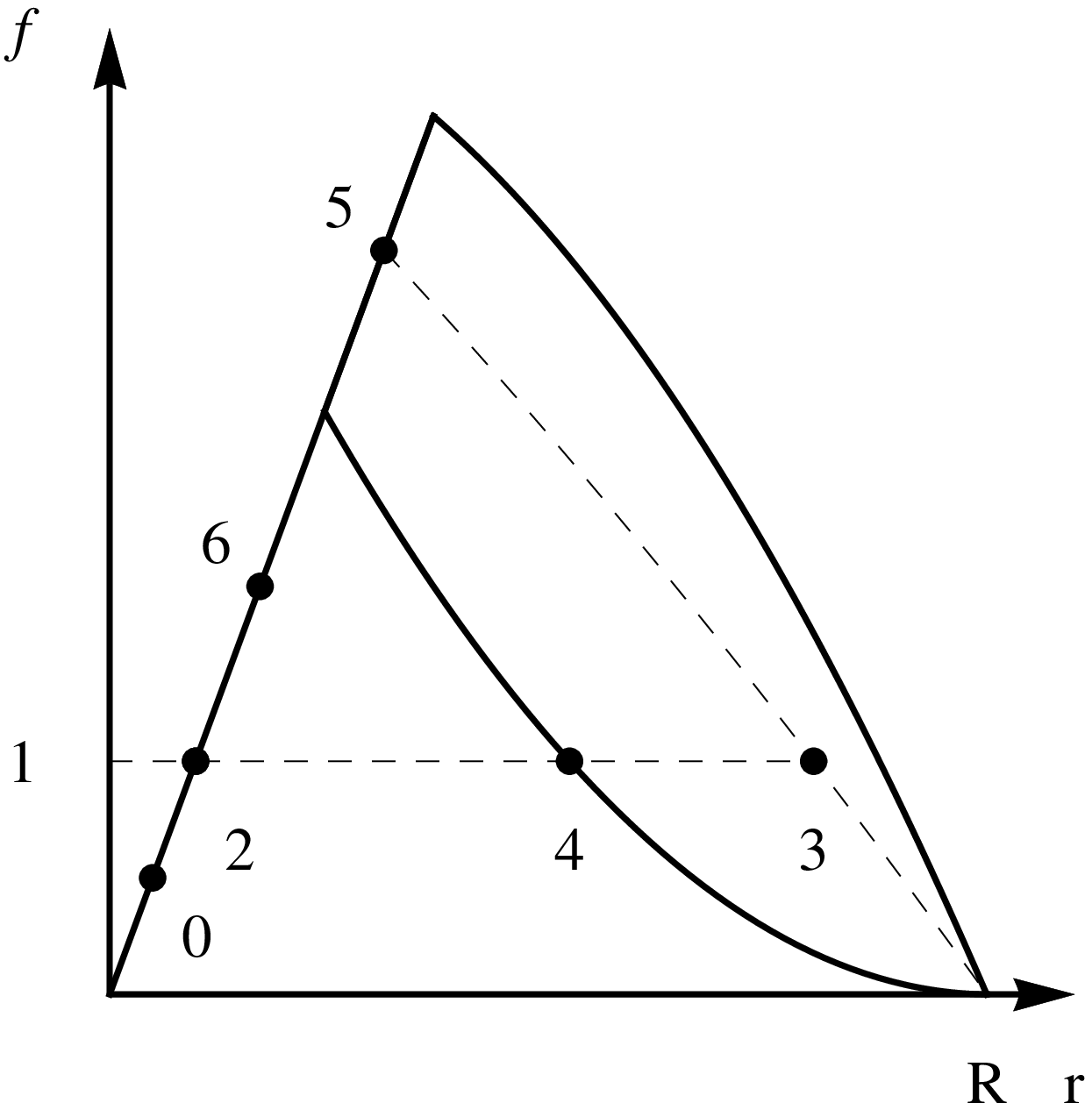}
    \end{psfrags}
        \caption{$(u_\ell,u_r) \in \Of^2$}
    \end{subfigure}
    \begin{subfigure}[b]{0.24\textwidth}
    \begin{psfrags}
      \psfrag{f}[c,c]{$f$}
      \psfrag{R}[c,B]{$R$}
      \psfrag{r}[c,B]{$\rho$}
      \psfrag{1}[c,B]{$F_1$}
      \psfrag{4}[c,B]{$F_2$}
      \psfrag{2}[l,b]{$\check{u}_1$}
      \psfrag{5}[c,b]{$\check{u}_2$}
      \psfrag{3}[c,b]{$\hat{u}_1$}
      \psfrag{6}[l,c]{$\vphantom{\int^{\int^\int}}\hat{u}_2~$}
      \psfrag{7}[c,b]{$u_*~$}
      \psfrag{8}[c,c]{$~u_r$}
      \psfrag{9}[c,B]{$u_\ell~$}
      \includegraphics[width=\textwidth]{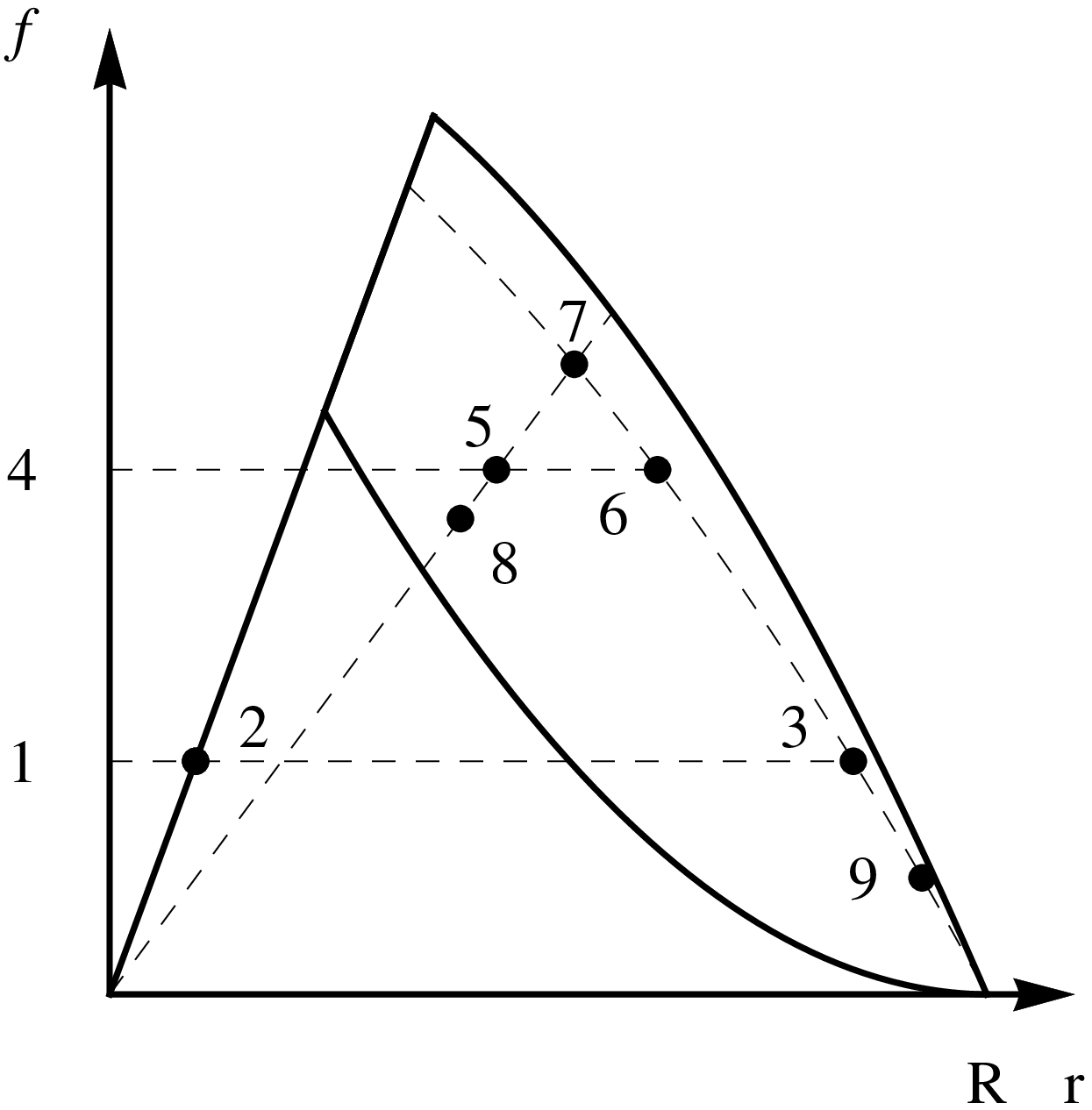}
    \end{psfrags}
        \caption{$(u_\ell,u_r) \in \Oc^2$}
    \end{subfigure}
    \begin{subfigure}[b]{0.24\textwidth}
    \begin{psfrags}
      \psfrag{f}[c,c]{$f$}
      \psfrag{R}[c,B]{$R$}
      \psfrag{r}[c,B]{$\rho$}
      \psfrag{1}[c,B]{$F$}
      \psfrag{5}[c,c]{$u_\ell~$}
      \psfrag{4}[l,c]{$u_r$}
      \psfrag{2}[c,b]{$\check{u}$}
      \psfrag{3}[c,b]{$\hat{u}~$}
      \includegraphics[width=\textwidth]{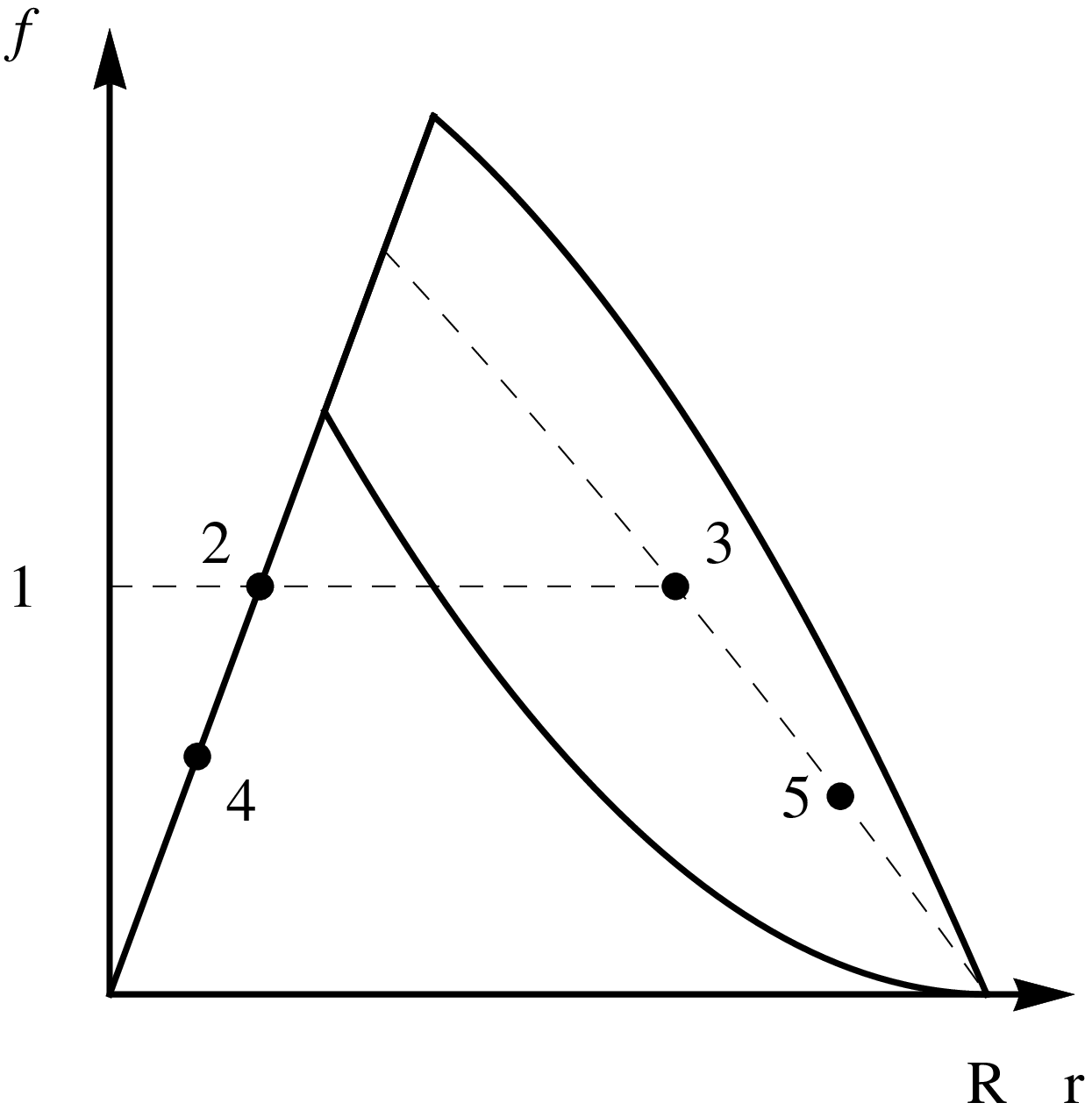}
    \end{psfrags}
        \caption{$(u_\ell,u_r) \in \Oc\times\Of^-$}
    \end{subfigure}
    \begin{subfigure}[b]{0.24\textwidth}
    \begin{psfrags}
      \psfrag{f}[c,c]{$f$}
      \psfrag{R}[c,B]{$R$}
      \psfrag{r}[c,B]{$\rho$}
      \psfrag{1}[c,B]{$F$}
      \psfrag{5}[c,b]{$u_\ell$}
      \psfrag{6}[c,B]{$u_r~$}
      \psfrag{2}[l,b]{$\check{u}$}
      \psfrag{3}[c,b]{$~\hat{u}$}
      \psfrag{4}[c,b]{}
      \includegraphics[width=\textwidth]{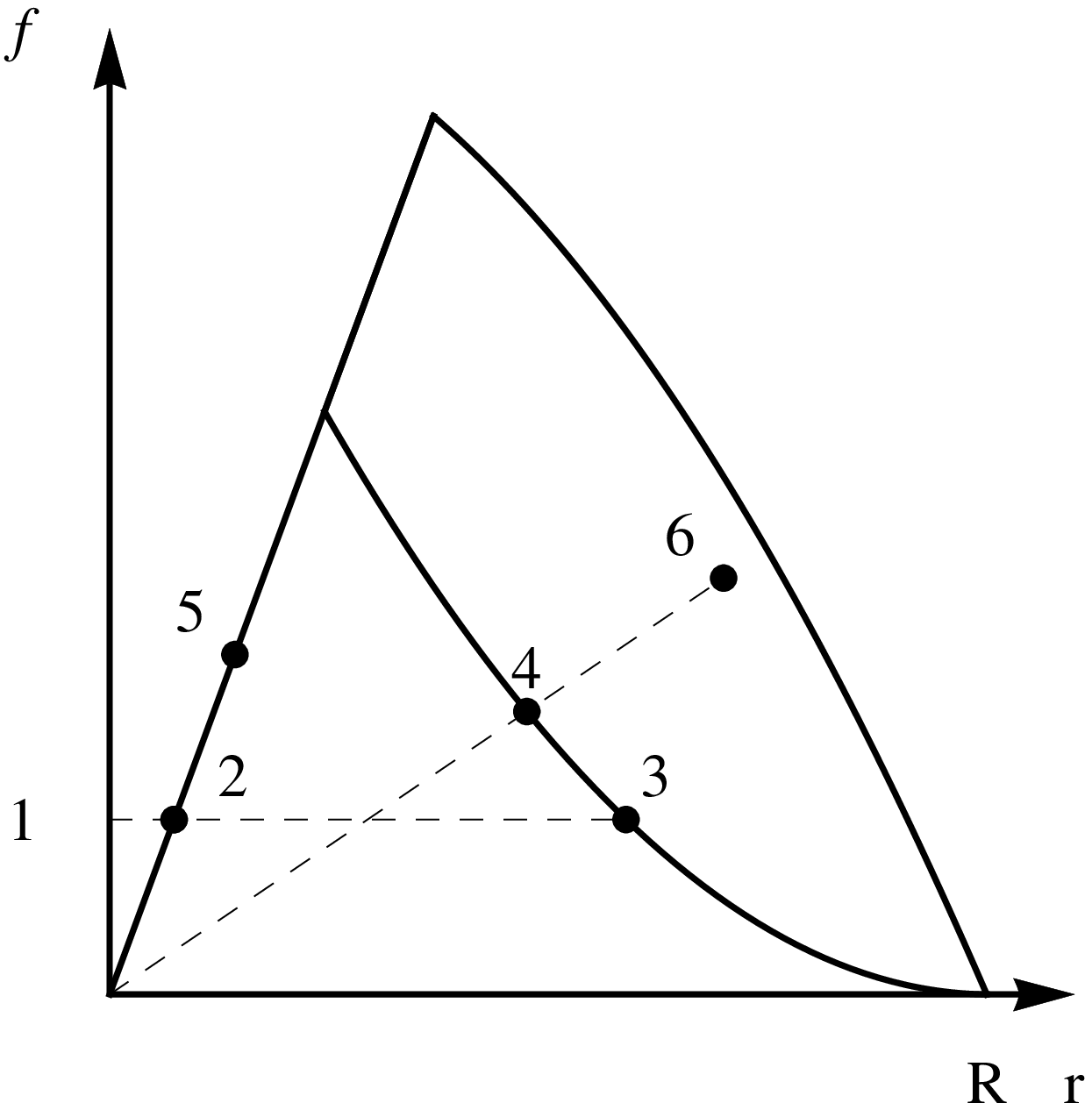}
    \end{psfrags}
        \caption{$(u_\ell,u_r) \in \Of^-\times\Oc$}
    \end{subfigure}
\caption{The selection criterion for $\hat{u}$ and $\check{u}$ given in Definition~\ref{def:01}.}
\label{fig:uhatucheck0}
\end{figure}

%
%
%
%
%

In the next two propositions we list the main properties of $\mathcal{R}_F$; the proofs are a case by case study and are deferred to Section~\ref{sec:tec1}.

\begin{proposition}\label{prop:cc}
The constrained Riemann solver $\mathcal{R}_F$ is $\Lloc1$-continuous and is not consistent, because it satisfies \eqref{P2} but not \eqref{P1} of Definition~\ref{def:cons}.
\end{proposition}

We conclude the section with some remarks on the invariant domains. 
Clearly, $\Omega$ is an invariant domain for both $\mathcal{R}$ and $\mathcal{R}_F$.
Moreover, $\Of$ and $\Oc$ are invariant domains for $\mathcal{R}$ but not for $\mathcal{R}_F$.
For this reason we look for minimal (w.r.t.\ inclusion) invariant domains for $\mathcal{R}_F$ containing $\Of$ or $\Oc$, see \figurename~\ref{fig:invariantSF1}.

\begin{proposition}\label{prop:IDR}
Let $\mathcal{R}_F$ be the constrained solver introduced in Definition~\ref{def:01}.

\begin{enumerate}[label={(IR.\arabic*)},leftmargin=*]\setlength{\itemsep}{0cm}%

\item\label{I11} The minimal invariant domain containing $\Of$ is $\mathcal{I}_f \doteq \Of \cup \mathcal{I}_1 \cup \mathcal{I}_2$, where
\begin{align*}
&\mathcal{I}_1 \doteq \bigl\{ u \in \Oc \,:\, f(u) \le F \le f(\psi
_2^+(u)) \bigr\},
&\mathcal{I}_2 \doteq \bigl\{ u \in \Oc \,:\, f(u) > F ,\ L_{w(u)}''(\rho) > 0 \bigr\}.
\end{align*}

\item\label{I12} The minimal invariant domain containing $\Oc$ is $\mathcal{I}_c \doteq \Oc \cup \left\{ \bigl( F/V, Q(F/V) \bigr)\right\}$.

\end{enumerate} 
\end{proposition}

\begin{figure}[ht]
\centering{
\begin{psfrags}
      \psfrag{f}[l,c]{$f$}
      \psfrag{F}[l,B]{$F$}
      \psfrag{R}[c,B]{$R$}
      \psfrag{r}[l,B]{$\rho$}
\includegraphics[width=.24\textwidth]{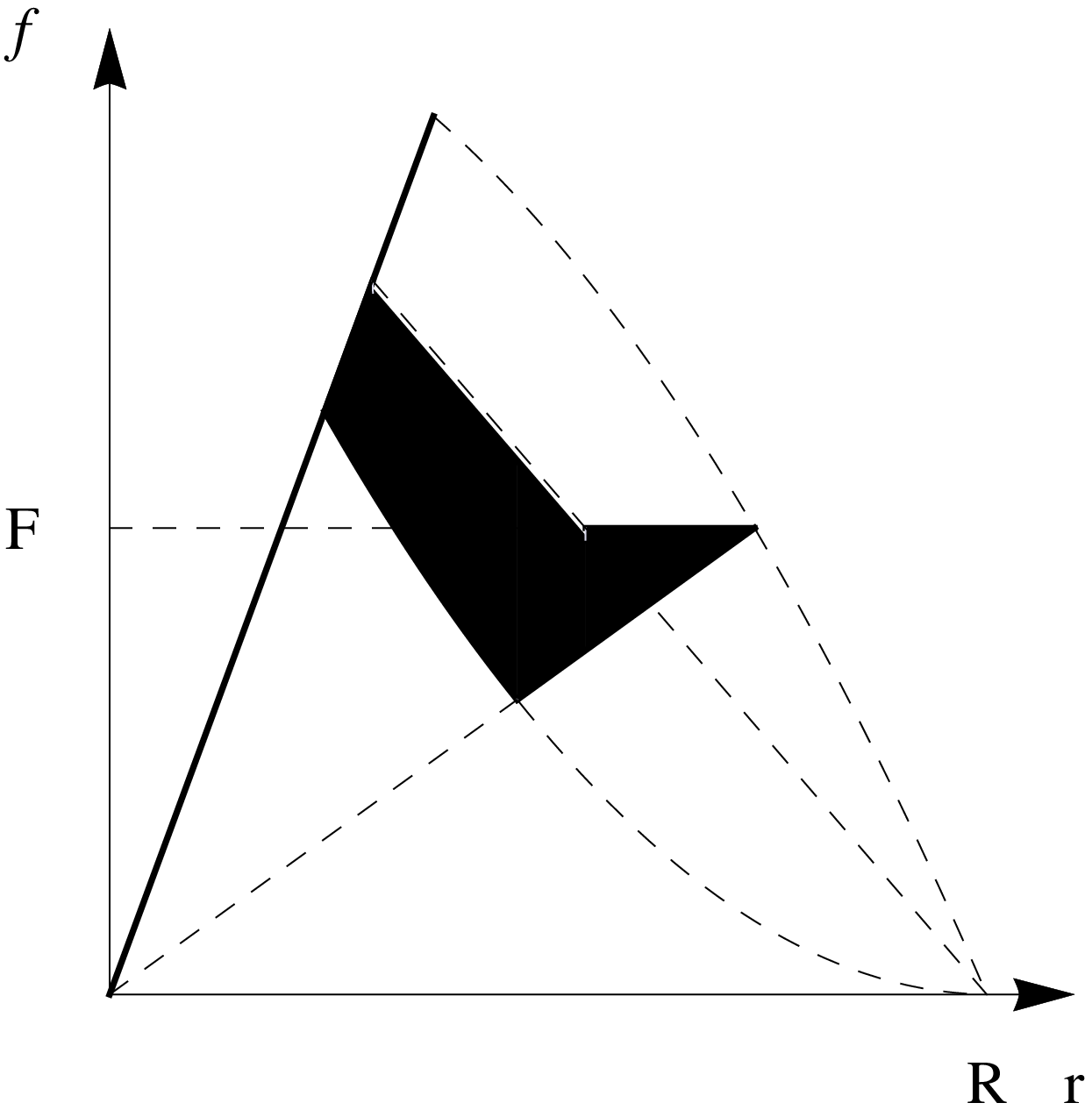}\qquad
\includegraphics[width=.24\textwidth]{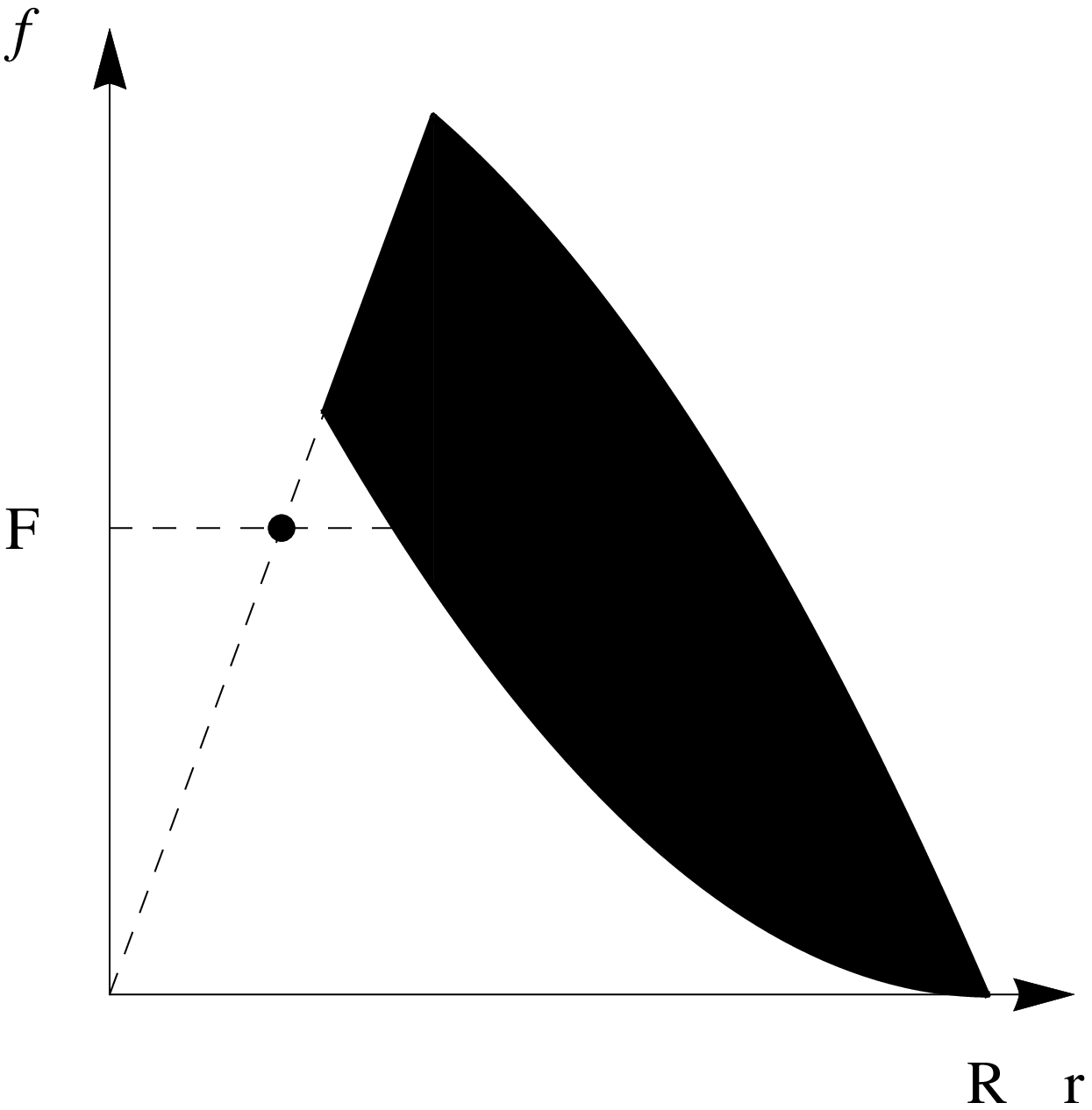}
\end{psfrags}}
\caption{Starting from the left, we represent $\mathcal{I}_f$ and $\mathcal{I}_c$ described respectively in \ref{I11} and \ref{I12} of Proposition~\ref{prop:IDR}.}
\label{fig:invariantSF1}
\end{figure}

\noindent We remark that $\mathcal{I}_2 = \emptyset$ for the PT$^p$ model, as well as for the PT$^a$ model in the case $L_{w_-}''(\sigma_-) \le 0$.


\section{The PT models with non-intersecting phases}\label{sec:noninter}


In this section we consider the case in which $\Of \cap \Oc = \emptyset$, namely the maximal velocities for the free and congested phases do not coincide, $V_c< V_f$.
This implies that a new kind of phase transition waves connecting $\Of^+$ and $\Oc$ appears, besides the ones from $\Of^-$ to $\Oc$. 
Below we give the definitions of the Riemann solver $\mathcal{S}$ and the constrained Riemann solver $\mathcal{S}_F$, which are valid both for the PT$^a$ and PT$^p$ models.

We remark that the analysis on the Riemann problem (with and without constraints) for the PT$^p$ model in the case of non-intersecting phases has already been carried out in \cite{BenyahiaRosini01,BenyahiaRosini02} and here it is understood in a more general framework.

\begin{definition}
The Riemann solver $\mathcal{S} \colon \Omega^2 \to \L\infty(\R;\Omega)$ associated to \eqref{eq:system},\eqref{eq:Rdata} is defined as follows.

\begin{enumerate}[label={(S.\arabic*)},leftmargin=*]\setlength{\itemsep}{0cm}%

\item
We let $\mathcal{S}[u_\ell,u_r]\doteq\mathcal{R}[u_\ell,u_r]$ whenever
\[\begin{array}{r@{}c@{\,}l@{\,}l}
(u_\ell, u_r)\in&& \Of^2 \cup \Oc^2 
&\cup
\{ (u_\ell , u_r) \in \Oc\times\Of \,:\, L_{w(u_\ell)}''(\rho_\ell) \ge 0 \}
\\[2pt]&&&\cup
\{ (u_\ell , u_r) \in \Of^-\times\Oc \,:\, \Lambda(u_\ell,u^c_-)\ge\lambda_1(u^c_-) \}
\\[2pt]&&&\cup
\{ (u_\ell , u_r) \in \Of^+\times\Oc \,:\, L_{w(u_\ell)}''(\rho_\ell) \le 0 \}.
\end{array}\]

\item\label{S2}
If $u_\ell\in\Oc$, $u_r\in \Of$ and $L_{w(u_\ell)}''(\rho_\ell) < 0$, then we let
\[
\mathcal{S}[u_\ell,u_r](x)\doteq 
\begin{cases}
\mathcal{R}[u_\ell,\psi_1^c(u_\ell)](x) &\text{for }x<\Lambda(\psi_1^c(u_\ell),\psi_1^f(u_\ell)),\\
\mathcal{R}[\psi_1^f(u_\ell),u_r](x) &\text{for }x>\Lambda(\psi_1^c(u_\ell),\psi_1^f(u_\ell)).
\end{cases} 
\]

\item\label{S3}
If $u_\ell\in\Of^-$, $u_r\in\Oc$ and $\Lambda(u_\ell,u^c_-)<\lambda_1(u^c_-)$, then we let
\[
\mathcal{S}[u_\ell,u_r](x)\doteq 
\begin{cases}
u_\ell&\text{for }x<\Lambda(u_\ell,u^c_-),\\
\mathcal{R}[u^c_-,u_r](x) &\text{for }x>\Lambda(u_\ell,u^c_-).
\end{cases} 
\]

\item\label{S4}
If $u_\ell \in \Of^+$, $u_r\in\Oc$ and $L_{w(u_\ell)}''(\rho_\ell) > 0$, then we let
\[
\mathcal{S}[u_\ell,u_r](x)\doteq
\begin{cases}
u_\ell&\text{for }x<\Lambda(u_\ell,\psi_1^c(u_\ell)),\\
\mathcal{R}[\psi_1^c(u_\ell),u_r](x) &\text{for }x>\Lambda(u_\ell,\psi_1^c(u_\ell)).
\end{cases}
\]

\end{enumerate}
\end{definition}

\begin{remark}\label{rem:specialRS}
Notice that $\mathcal{S}$ differs from $ \mathcal{R}$ (corresponding to $V \doteq V_f$) only in the cases described in \ref{S2}, \ref{S3} and \ref{S4}, namely $\mathcal{S}[u_\ell,u_r]$ differs from $ \mathcal{R}[u_\ell,u_r]$ if and only if $(u_\ell,u_r)$ satisfies one of the following conditions:
\begin{align}\label{eq:specialRS1}
&u_\ell \in \Oc,&
&u_r \in \Of,&
&L''_{w(u_\ell)}(\rho_\ell) <0,
\\\label{eq:specialRS2}
&u_\ell \in \Of^+,&
&u_r \in \Oc,&
&L''_{w(u_\ell)}(\rho_\ell) >0,
\\\label{eq:specialRS3}
&u_\ell \in \Of^-,&
&u_r \in \Oc,&
&\Lambda(u_\ell,u_-^c) < \lambda_1(u_-^c).
\end{align}
In particular, for the PT$^p$ model we have that $\mathcal{S}[u_\ell,u_r]$ differs from $ \mathcal{R}[u_\ell,u_r]$ (corresponding to $V \doteq V_f$) if and only if $(u_\ell, u_r)\in \Oc\times\Of$; this is also the case for the PT$^a$ model if $L_{w_-}''(\sigma_-^f) < 0$.
\end{remark}

In the next proposition we list the main properties of $\mathcal{S}$; the proof is a case by case study and is deferred to Section~\ref{sec:tec2}.

\begin{proposition}\label{prop:cc2}
The Riemann solver $\mathcal{S}$ is $\Lloc1$-continuous and consistent.
\end{proposition}

Before introducing the Riemann solver $\mathcal{S}_F$, we observe that in the present case
\[\begin{array}{r@{}c@{\,}l}
\mathcal{D}_1 = &&
\{ (u_\ell , u_r) \in \Of^2 \,:\, f(u_\ell) \le F \}
\cup
\{ (u_\ell , u_r) \in \Oc^2 \,:\, f(u_*(u_\ell,u_r)) \le F \}
\\[2pt]&\cup&
\{ (u_\ell , u_r) \in \Oc \times \Of \,:\, f(\psi_1^f(u_\ell)) \le F \}
\cup
\{ (u_\ell , u_r) \in \Of^- \times \Oc \,:\, \min\{f(u_\ell),f(\psi_2^-(u_r))\} \le F \}
\\[2pt]&\cup&
\{ (u_\ell , u_r) \in \Of^+ \times \Oc \,:\, f(u_*(u_\ell,u_r)) \le F \}.
\end{array}\]

\begin{figure}
    \centering
    \begin{subfigure}[b]{0.27\textwidth}
    \begin{psfrags}
      \psfrag{f}[c,c]{$f$}
      \psfrag{R}[c,B]{$R$}
      \psfrag{r}[l,B]{$\rho$}
      \psfrag{4}[c,B]{$F$}
      \psfrag{6}[c,c]{$\check{u}~$}
      \psfrag{7}[l,B]{$\hat{u}$}
      \psfrag{5}[c,b]{$u_\ell\,$}
      \psfrag{8}[l,c]{$u_r$}
      \includegraphics[width=0.88\textwidth]{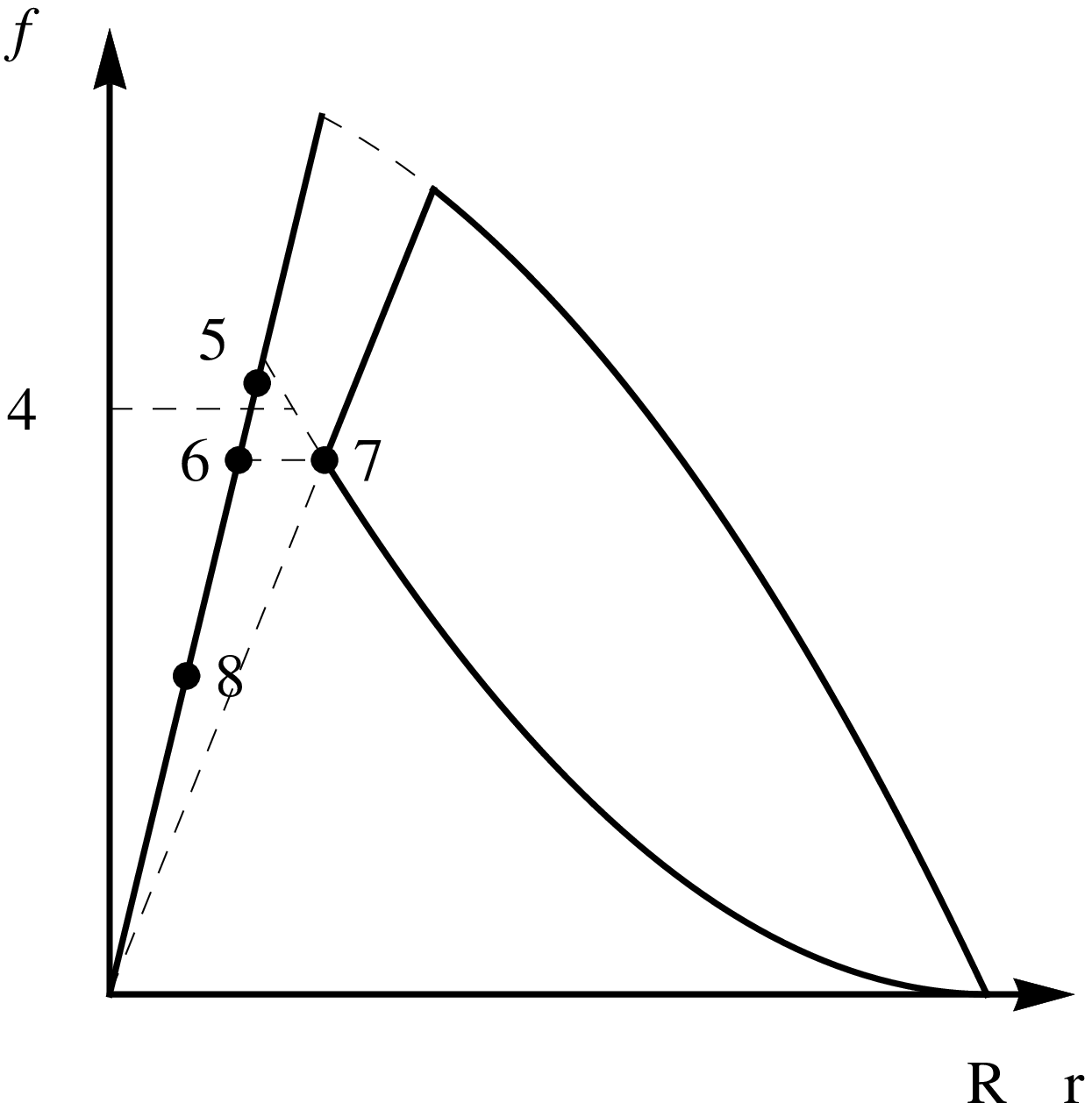}
    \end{psfrags}
        \caption{Case \eqref{eq:special1}.}
    \end{subfigure}
    \begin{subfigure}[b]{0.27\textwidth}
    \begin{psfrags}
      \psfrag{f}[c,c]{$f$}
      \psfrag{R}[c,B]{$R$}
      \psfrag{r}[l,B]{$\rho$}
      \psfrag{4}[c,B]{$F$}
      \psfrag{6}[c,c]{$\check{u}~$}
      \psfrag{7}[c,c]{$u_\ell~$}
      \psfrag{5}[l,B]{$\hat{u}$}
      \psfrag{8}[l,c]{$u_r$}
      \includegraphics[width=0.88\textwidth]{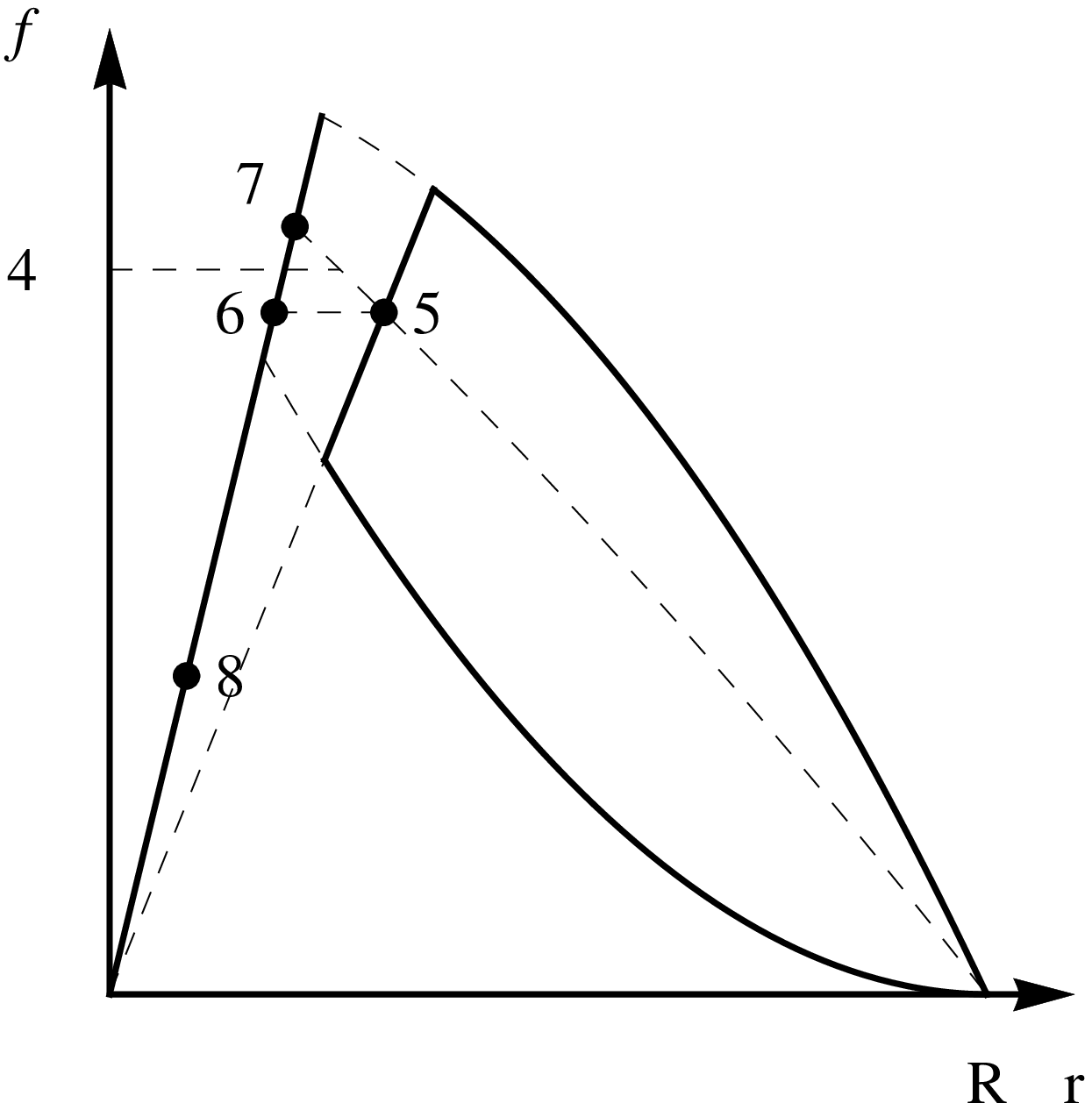}
    \end{psfrags}
        \caption{Case \eqref{eq:special2}.}
    \end{subfigure}
    \begin{subfigure}[b]{0.27\textwidth}
    \begin{psfrags}
      \psfrag{f}[c,c]{$f$}
      \psfrag{R}[c,B]{$R$}
      \psfrag{r}[l,B]{$\rho$}
      \psfrag{4}[c,B]{$F$}
      \psfrag{6}[c,c]{$\check{u}~$}
      \psfrag{9}[c,c]{$u_\ell$}
      \psfrag{7}[c,B]{$\psi_1^\ell~$}
      \psfrag{5}[l,B]{$\hat{u}$}
      \psfrag{8}[l,c]{$u_r$}
      \includegraphics[width=0.88\textwidth]{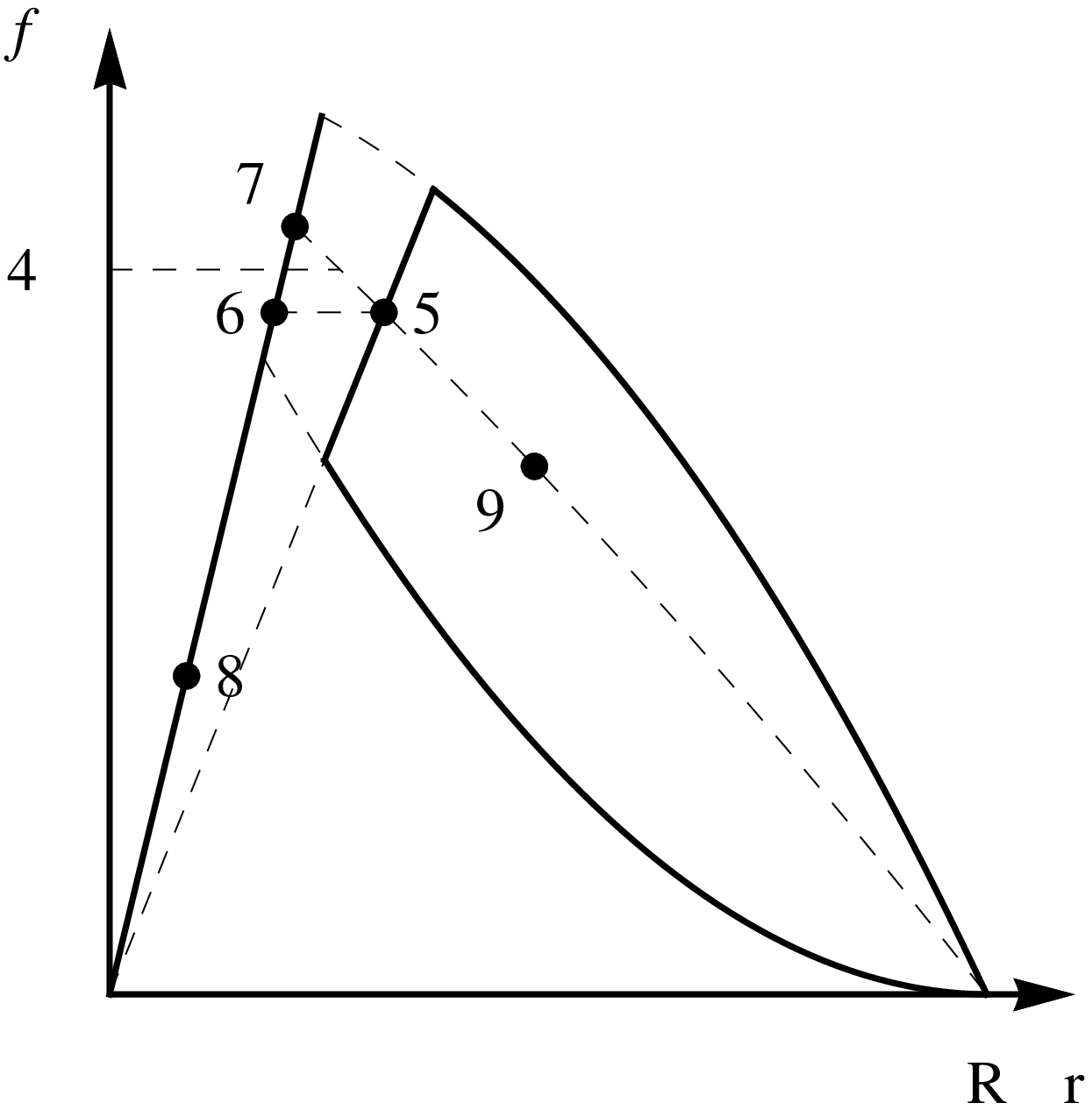}
    \end{psfrags}
        \caption{Case \eqref{eq:special3}, $\psi_1^\ell \doteq \psi_1^f(u_\ell)$.}
    \end{subfigure}
\caption{The selection criterion for $\hat{u}$ and $\check{u}$ given in Definition~\ref{def:04}.}
\label{fig:uhatucheck11}
\end{figure}

\begin{definition}\label{def:04}
The constrained Riemann solver $\mathcal{S}_F \colon \Omega^2 \to \L\infty(\R;\Omega)$  associated to \eqref{eq:system},\eqref{eq:Rdata},\eqref{eq:const} is defined as
\[
\mathcal{S}_F[u_\ell,u_r](x)\doteq \begin{cases}
\mathcal{S}[u_\ell,u_r](x) &\hbox{if }(u_\ell,u_r) \in \mathcal{D}_1,\\[5pt]
\begin{cases}
\mathcal{S}[u_\ell,\hat{u}](x) &\hbox{if }x<0,\\
\mathcal{S}[\check{u},u_r](x) &\hbox{if }x>0,
\end{cases}
&\hbox{if }(u_\ell,u_r) \in \mathcal{D}_2,
\end{cases}
\]
where $\hat{u}  = \hat{u}(u_\ell,F) \in \Oc$ and $\check{u}=\check{u}(u_\ell,u_r,F)  \in \Omega$ are uniquely selected by the conditions 
\begin{gather*}
f(\hat{u})=f(\check{u})=\max \bigl\{f(u)\le F \,:\, u\in\Oc, \, w(u)=\max\{w(u_\ell),w_-\} \bigr\},
\\
w(\hat{u})=\max\{w(u_\ell),w_-\},
\qquad
v(\check{u})=\begin{cases}
V_f&\hbox{if }f(\psi_2^-(u_r))> F,\\
v_r &\hbox{if }f(\psi_2^-(u_r))\le F.
\end{cases}
\end{gather*}
\end{definition}

\noindent
\begin{remark}\label{rem:special}
Notice that, if $(u_\ell,u_r) \in \mathcal{D}_2$, then the selection criterion for $\hat{u}$ and $\check{u}$ given above does not coincide with the one given in Definition~\ref{def:01} with $V \doteq V_f$ if and only if $(u_\ell,u_r)$ satisfies one of the following conditions
\begin{align}\label{eq:special1}
&u_\ell \in \Of^-,&
&u_r \in \Of,&
&F \in \bigl(f(u_-^c), f(u_\ell)\bigr),
\\\label{eq:special2}
&u_\ell \in \Of^+,&
&u_r \in \Of,&
&F \in \bigl(f(\psi_1^c(u_\ell)), f(u_\ell)\bigr),
\\\label{eq:special3}
&u_\ell \in \Oc,&
&u_r \in \Of,&
&F \in \bigl(f(\psi_1^c(u_\ell)), f(\psi_1^f(u_\ell))\bigr),
\end{align}
and in this case $f(\hat{u}) = f(\check{u}) < F$ and $v(\hat{u})=V_c$.
For this reason, in \figurename~\ref{fig:uhatucheck11} we specify the selection criterion for $\hat{u}$ and $\check{u}$ given above only in these cases.

As a consequence, we have that $\mathcal{R}_F[u_\ell,u_r] \ne \mathcal{S}_F[u_\ell,u_r]$ if and only if $(u_\ell,u_r) \in \mathcal{D}_1$ and satisfies one of the conditions \eqref{eq:specialRS1},\eqref{eq:specialRS2},\eqref{eq:specialRS3}, or $(u_\ell,u_r) \in \mathcal{D}_2$ and satisfies \eqref{eq:special1},\eqref{eq:special2},\eqref{eq:special3}.
\end{remark}
\noindent
Clearly, the general properties of $\hat{u}$ and $\check{u}$ listed in \eqref{gen.properties} are satisfied also in the present case.

In the next proposition we state the main properties of $\mathcal{S}_F$; the proof is a case by case study and is deferred to Section~\ref{sec:tec3}.

\begin{proposition}\label{prop:cc3}
The Riemann solver $\mathcal{S}_F$ is neither $\Lloc1$-continuous nor consistent, because it satisfies \eqref{P2} but not \eqref{P1} of Definition~\ref{def:cons}.
\end{proposition}

The next proposition is devoted to the minimal invariant domains for $\mathcal{S}_F$.

\begin{proposition}
Let $\mathcal{S}_F$ be the solver introduced in Definition~\ref{def:04}.

\begin{enumerate}[label={(IS.\arabic*)},leftmargin=*]\setlength{\itemsep}{0cm}%

\item 
The minimal invariant domain containing $\Of$ is $\mathcal{I}_f$ defined in \ref{I11} of Proposition~\ref{prop:IDR}.

\item 
The minimal invariant domain containing $\Oc$ is
\[
\mathcal{I}_c \doteq 
\begin{cases}
\Oc \cup \left\{ \bigl( F/V_f, Q(F/V_f) \bigr)\right\} &\text{if } F<f(u^c_-),\\
\Oc &\text{if } F\ge f(u^c_-).
\end{cases}
\]
\end{enumerate} 
\end{proposition}

\noindent The proof is omitted since it is analogous to that of Proposition \ref{prop:IDR}.



\section{Total variation estimates}\label{sec:tv}


Consider any of the PT models previously introduced. 
Fix a couple $(u_\ell,u_r)$ of initial states in $\Omega^2$ and consider $u_1 \doteq \mathcal{R}_F[u_\ell,u_r]$ and $u_2\doteq \mathcal{S}_F[u_\ell,u_r]$.
In this section, we study the total variation of $u_1$ and $u_2$ in the $v$ and $w$ coordinates.
More precisely, for $i=1,2$ we study the sign of
\begin{align*}
&\Delta\tv_v^i \doteq \tv(v(u_i)) - |v(u_\ell)-v(u_r)|,
&\Delta\tv_w^i \doteq \tv(w(u_i)) - |w(u_\ell)-w(u_r)|.
\end{align*}
This study may be useful to compare the difficulty of applying the two Riemann solvers in a wave-front tracking scheme; see \cite{HoldenRisebroWFT} and the references therein.
The choice of the Riemann invariant coordinates (rather than the conserved variables) stems from the fact that in these coordinates the total variation of both $\mathcal{R}[u_\ell,u_r]$ and $\mathcal{S}[u_\ell,u_r]$ does not increase; see \cite{AndreianovDonadelloRazafisonRollandRosiniNHM2016, AndreianovDonadelloRosiniM3ASS2016} where this property is exploited to prove existence results for the ARZ model and \cite{BenyahiaRosini01} for the PT$^p$ model.

\begin{proposition}\label{prop:tvestimate}
If $(u_\ell, u_r)$ is a couple of initial states in $\Omega^2$, then for $i=1,2$ we have that $\Delta\tv_w^i = 0 = \Delta \tv_v^i$ if and only if $(u_\ell, u_r)$ belongs to
\begin{align*}
\mathcal{D}_1 &\cup
\{(u_\ell, u_r) \in \Oc^2 \,:\, f(\psi_2^-(u_r)) \le F,\ v(u_\ell) \le v(\hat{u}),\ w(u_r) \le w(\check{u})\}
\\&\cup
\{(u_\ell, u_r) \in \Oc^- \times \Of^- \,:\, v(u_\ell) \le v(\hat{u}),\ w(u_r) \le w(\check{u})\}.
\end{align*}
In all the other cases, $\Delta\tv_w^i$ and $\Delta \tv_v^i$ are non-negative and $\Delta\tv_w^i + \Delta \tv_v^i$ is strictly positive.
\end{proposition}
\begin{proof}
It is easy to see that $\Delta\tv_w = 0 = \Delta \tv_v$ for all $(u_\ell,u_r)\in\mathcal{D}_1$.
Fix therefore $(u_\ell, u_r) \in \mathcal{D}_2$.
We consider only the case $i=1$, being the case $i=2$ analogous.
For notational simplicity we drop the superscript and let $(v_\ell, w_\ell) \doteq (v(u_\ell), w(u_\ell))$, $(v_r, w_r) \doteq (v(u_r), w(u_r))$, $\hat{v} \doteq v(\hat{u})$ and $\check{w} \doteq w(\check{u})$.
According to Definition~\ref{def:01} we have to distinguish the following cases.

\begin{itemize}[itemindent=*,leftmargin=0pt]\setlength{\itemsep}{0cm}%

\item 
If $u_\ell \in \Of^-$ and $u_r \in \Of$, then
\begin{align*}
\Delta\tv_v &= 2(V-\hat{v}) > 0,&
\Delta\tv_w &= 2 w_--w_\ell-\check{w}+|w_r-\check{w}| - |w_r-w_\ell|
\ge
2(w_--w_\ell)>0.
\end{align*}

\item 
If $u_\ell \in \Of^+$ and $u_r \in \Of$, then
\begin{align*}
\Delta\tv_v &= 2(V-\hat{v}) > 0,&
\Delta\tv_w &= w_\ell-\check{w}+|w_r-\check{w}| - |w_r-w_\ell|
\ge0.
\end{align*}

\item 
If $u_\ell,u_r\in\Oc$ and $f(\psi_2^-(u_r)) > F$, then
\begin{align*}
\Delta\tv_v &= |v_\ell-\hat{v}|+ 2V-\hat{v}-v_r - |v_\ell-v_r| \ge 2(V-v_r) \ge 0,\\
\Delta\tv_w &= w_\ell + w_r - 2\check{w} - |w_r-w_\ell|
=2\left(\min\{w_\ell,w_r\} - \check{w}\right) >0.
\end{align*}

\item 
If $u_\ell,u_r\in\Oc$ and $f(\psi_2^-(u_r)) \le F$, then
\begin{align*}
\Delta\tv_v &= |v_\ell-\hat{v}| + (v_r-\hat{v}) - |v_r-v_\ell| \ge 0,&
\Delta\tv_w &= (w_\ell-\check{w})+|w_r-\check{w}| - |w_r-w_\ell| \ge 0.
\end{align*}

\item 
If $u_\ell\in \Oc^-$ and $u_r\in \Of^-$, then  
\begin{align*}
\Delta\tv_v&=|v_\ell-\hat{v}|+ v_\ell-\hat{v}\ge0,&
\Delta\tv_w&=(w_\ell-\check{w})+|\check{w}-w_r|- |w_r-w_\ell| \ge 0.
\end{align*}

\item 
If $u_\ell\in \Of^-$ and $u_r\in \Oc^-$, then 
\begin{align*}
\Delta\tv_v&=2(V-\hat{v})>0,& 
\Delta\tv_w=2(w_--\check{w})>0.&\qedhere
\end{align*}
\end{itemize}
\end{proof}

In the next proposition, we compare the total variation of $u_1$ with that of $u_2$.
To do so we have to fix $(u_\ell, u_r)$ in the intersection of the domains of definition for both $\mathcal{R}_F$ and $\mathcal{S}_F$.
However, under the natural assumption that $V = V_f > V_c$, the domain of definition for $\mathcal{R}_F$ strictly contains that for $\mathcal{S}_F$.

\begin{proposition}
Let $V = V_f > V_c$.
If $(u_\ell,u_r)$ belongs to the domain of definition of $\mathcal{S}_F$, then $\Delta \tv_{v,w}^1 \le \Delta \tv_{v,w}^2$.
Moreover, we have:
\begin{itemize}[itemindent=*,leftmargin=0pt]\setlength{\itemsep}{0cm}%

\item 
$\Delta \tv_{v}^1 < \Delta \tv_{v}^2$ if and only if $(u_\ell, u_r)$  satisfies one of the conditions \eqref{eq:special1},\eqref{eq:special2};

\item
$\Delta \tv_{w}^1 < \Delta \tv_{w}^2$ if and only if $(u_\ell, u_r)$ is such that $w_r > w(\check{u}_2)$ and it satisfies one of the conditions \eqref{eq:special1},\eqref{eq:special2},\eqref{eq:special3}.
\end{itemize}
\end{proposition}

\begin{proof}
By Proposition~\ref{prop:tvestimate} and Remark~\ref{rem:special}, it is sufficient to consider the following cases.
For brevity, we let $(v_\ell, w_\ell) \doteq (v(u_\ell), w(u_\ell))$, $(v_r, w_r) \doteq (v(u_r), w(u_r))$, $\hat{v}_i \doteq v(\hat{u}_i)$ and $\check{w}_i \doteq w(\check{u}_i)$ for $i=1,2$.
\begin{itemize}[itemindent=*,leftmargin=0pt]\setlength{\itemsep}{0cm}%

\item 
If $(u_\ell, u_r)$ satisfies \eqref{eq:special1} or \eqref{eq:special2}, then $\hat{v}_1 > \hat{v}_2 = V_c$, $\check{w}_1 > \check{w}_2$ and
\begin{align*}
\Delta\tv_v^1 - \Delta\tv_v^2 &= 2(V_c-\hat{v}_1) < 0,&
\Delta\tv_w^1 - \Delta\tv_w^2 &= \check{w}_2 - \check{w}_1 + |\check{w}_1 - w_r| - |\check{w}_2 - w_r| \le 0.
\end{align*}

\item 
If $(u_\ell, u_r)$ satisfies \eqref{eq:special3}, then $\hat{v}_1 > \hat{v}_2 = V_c$, $\check{w}_1 > \check{w}_2$ and
\begin{align*}
\Delta\tv_v^1 = \Delta\tv_v^2 & = 0,&
\Delta\tv_w^1 - \Delta\tv_w^2 &= \check{w}_2 - \check{w}_1 + |\check{w}_1 - w_r| - |\check{w}_2 - w_r| \le 0.
\end{align*}

\end{itemize}
It is easy to see that in both cases we have $\Delta\tv_w^1 - \Delta\tv_w^2 = 0$ if and only if $w_r \le \check{w}_2$.
\end{proof}

\section{Numerical example}\label{sec:simu01}


In this section, we apply the Riemann solvers $\mathcal{R}$ and $\mathcal{R}_F$ introduced in Section~\ref{sec:inter} to simulate the traffic across a toll gate.
More specifically, we consider a toll gate placed in $x=0$ and two types of vehicles: the $1$-vehicles characterized by the Lagrangian marker $w_1$ and the $2$-vehicles characterized by the Lagrangian marker $w_2$.
Assume that the $1$-vehicles and the $2$-vehicles are initially stopped and uniformly distributed respectively in $(x_1,x_2)$ and $(x_2,0)$, with $x_1<x_2<0$.
If $F \in (0,V_f \, \sigma_+^{f})$ is the capacity of the toll gate, then the resulting model is given by the Cauchy problem for \eqref{eq:system},\eqref{eq:const}, with piecewise constant initial datum
\[
u(0,x) =
\begin{cases}
u_1&\text{if }x \in (x_1,x_2),
\\
u_2&\text{if }x \in (x_2,0),
\\
(0,0)&\text{otherwise},
\end{cases}
\]
where $u_i \in \Omega_c$ is such that $v(u_i) = 0$ and $w(u_i) = w_i$ for $i=1,2$.

While the overall picture of the corresponding solution is rather stable, a detailed analytical study needs to consider many slightly different cases. 
Below, we restrict the construction of the solution corresponding to the PT$^0$ model with $V_f = V_c \doteq V$ and consider the situation where $w_{2}<0<w_{1}$ and $F \in (0,V\,\sigma_-^f)$. 
\begin{figure}
\centering{
\begin{psfrags}
      \psfrag{a}[l,c]{$f$}
      \psfrag{f}[l,B]{$F$}
      \psfrag{b}[l,B]{$\rho$}
      \psfrag{r}[c,B]{$R$}
      \psfrag{c}[c,c]{$\hat{u}_{1}~$}
      \psfrag{d}[l,c]{$\,\check{u}$}
      \psfrag{e}[l,B]{$\hat{u}_{2}$}
      \psfrag{g}[l,B]{$u_*$}
\includegraphics[width=.27\textwidth]{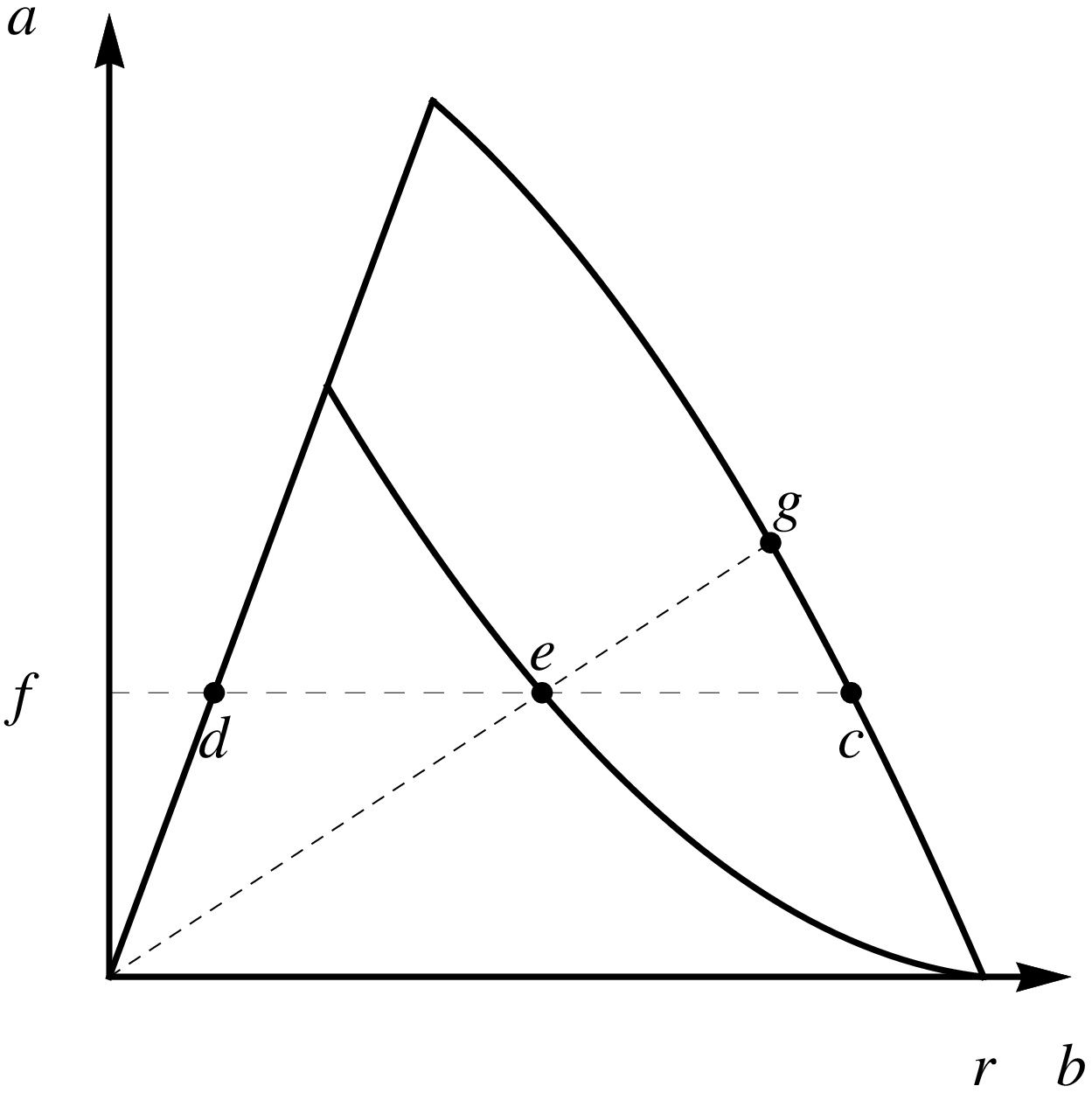}\qquad
      \psfrag{a}[c,B]{$x$}
      \psfrag{b}[l,c]{$t$}
      \psfrag{c}[c,B]{$x_{2}$}
      \psfrag{d}[c,B]{$x_{1}$}
      \psfrag{A}[c,b]{$a_{1}$}
      \psfrag{B}[c,B]{$a_{2}$}
      \psfrag{C}[r,b]{$a_{3}$}
      \psfrag{D}[l,c]{$a_{4}$}
      \psfrag{F}[r,b]{$a_{5}$}
      \psfrag{E}[l,c]{$a_{6}$}
\includegraphics[width=.27\textwidth]{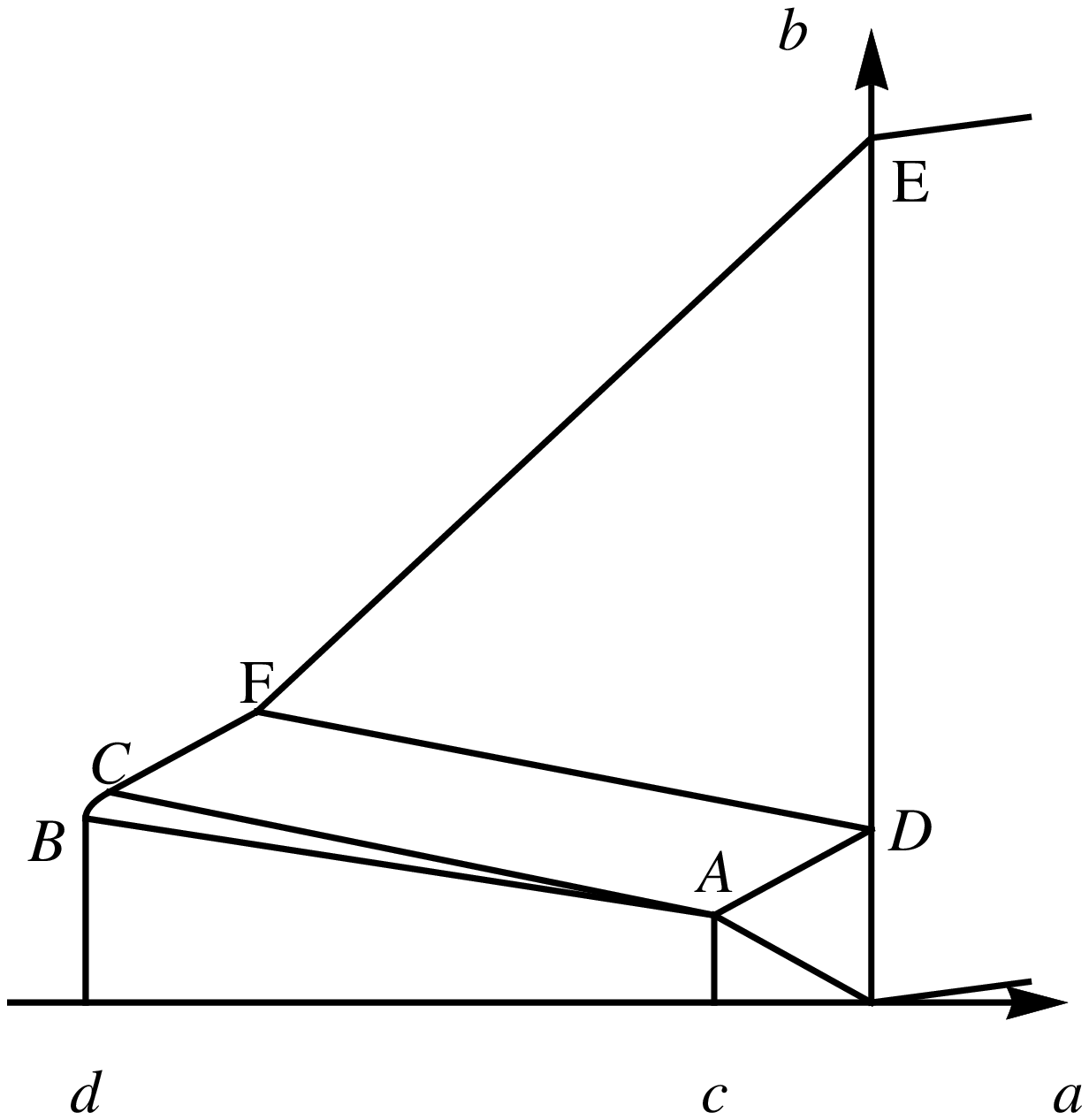}\qquad
      \psfrag{a}[c,c]{$x$}
      \psfrag{b}[c,c]{$t$}
      \psfrag{c}[c,B]{$x_{2}$}
      \psfrag{d}[c,B]{$x_{1}$}
      \psfrag{C}[c,B]{\tiny$\mathsf{PT}_{1}\quad$}
      \psfrag{D}[c,B]{\tiny$\mathsf{PT}_{2}\quad\,$}
      \psfrag{O}[c,B]{\tiny$\mathsf{PT}_{3}\quad$}
      \psfrag{G}[c,B]{\tiny$\mathsf{PT}_{4}\quad$}
      \psfrag{A}[c,c]{\tiny$\mathsf{C}_{1}$}
      \psfrag{J}[l,c]{\tiny$\mathsf{C}_{2}$}
      \psfrag{N}[c,B]{\tiny$\mathsf{C}_{3}$}
      \psfrag{Z}[c,b]{\tiny$\quad\mathsf{C}_{4}$}
      \psfrag{B}[l,B]{\tiny$\mathsf{S}_{1}$}
      \psfrag{F}[c,B]{\tiny$\mathsf{S}_{2}$}
      \psfrag{L}[l,c]{\tiny$\mathsf{U}_{1}$}
      \psfrag{M}[l,c]{\tiny$\mathsf{U}_{2}$}
      \psfrag{R}[c,c]{\tiny$\mathsf{R}$}
\includegraphics[width=.27\textwidth]{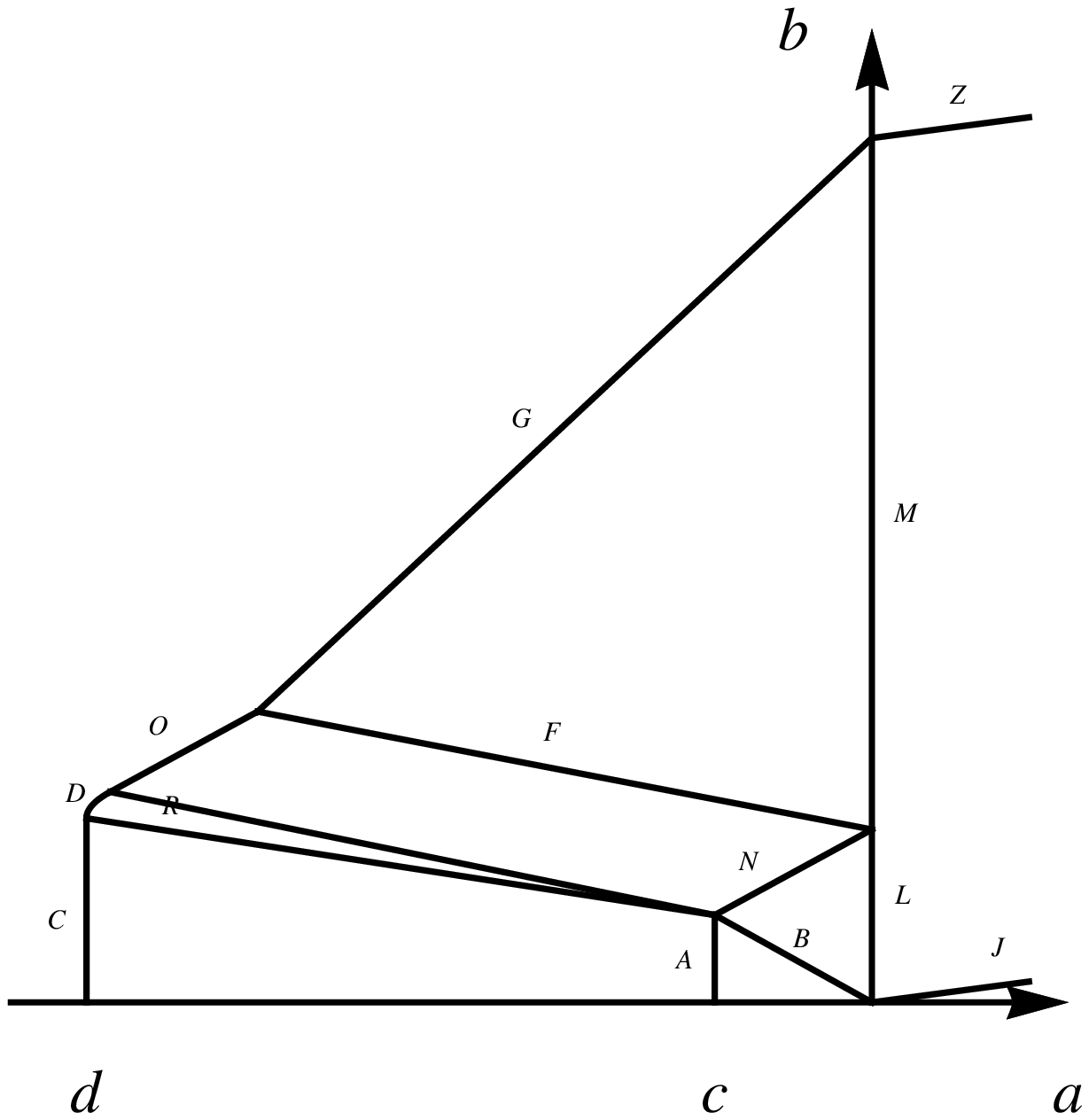}
\end{psfrags}}
\caption{The solution constructed in Section~\ref{sec:simu01} and corresponding to the numerical data~\eqref{nume.data}.}
\label{fig:simu01}
\end{figure}
In particular, this implies that $\rho \mapsto L_{w_{1}}(\rho)$ is concave and $\rho \mapsto L_{w_{2}}(\rho)$ is convex, see \figurename~\ref{fig:simu01}, left.
Moreover, for simplicity we normalize the maximal density and take $R=1$.

The solution is constructed by applying the wave-front tracking method \cite{HoldenRisebroWFT} based on the Riemann solver $\mathcal{R}$ away from $x = 0$ and on the constrained Riemann solver $\mathcal{R}_F$ at $x = 0$.
We use the following notation
\begin{align*}
&\hat{u}_1 \doteq \hat{u}(u_1,F),&
&\hat{u}_2 \doteq \hat{u}(u_2,F),&
&\check{u} \doteq \check{u}((0,0),F),&
&u_* \doteq u_*(u_1,\hat{u}_2).
\end{align*}
The first step in the construction of the solution is solving the Riemann problems at the points $(x,t) \in \{ (x_{1},0), (x_{2},0),(0,0)\}$. 
\begin{itemize}[leftmargin=*]\setlength{\itemsep}{0cm}%

\item The Riemann problem at $(x_{1},0)$ is solved by a stationary phase transition $\mathsf{PT}_{1}$ from $(0,0)$ to $u_1$.

\item The Riemann problem at $(x_{2},0)$ is solved by a stationary contact discontinuity $\mathsf{C}_{1}$ from $u_1$ to $u_2$.

\item The Riemann problem at $(0,0)$ is solved by a shock $\mathsf{S}_{1}$ from $u_2$ to $\hat{u}_{2}$ travelling with negative speed $\Lambda(u_2,\hat{u}_{2})$, a stationary undercompressive shock $\mathsf{U}_1$ from $\hat{u}_{2}$ to $\check{u}$, and a contact discontinuity $\mathsf{C}_2$ from $\check{u}$ to $(0,0)$ travelling with speed $V$.
\end{itemize}
To prolong the solution, we have to consider the Riemann problems arising at each interaction as follows.

\begin{itemize}[leftmargin=*]\setlength{\itemsep}{0cm}%

\item First,  $\mathsf{C}_{1}$ interacts with $\mathsf{S}_{1}$ at $a_{1} \doteq (x_2, t_{a_{1}})$, where $t_{a_{1}} \doteq x_{2}/\Lambda(u_2,\hat{u}_{2})$. The Riemann problem at $a_{1}$ is solved by a rarefaction $\mathsf{R}$ from $u_1$ to $u_*$ and a contact discontinuity $\mathsf{C}_{3}$ from $u_*$ to $\hat{u}_{2}$ travelling with speed $v(\hat{u}_2)$. 
The rarefaction $\mathsf{R}$ has support in the cone
\[
\mathcal{C} \doteq
\{ (x,t) \in \mathbb{R} \times \mathbb{R_{+}} \,:\, \lambda_{1}(u_1) (t-t_{a_{1}}) \leq x-x_{2} \leq \lambda_{1}(u_*) (t-t_{a_{1}}), ~ t \ge t_{a_{1}} \}.
\]

\item
$\mathsf{C}_{3}$ interacts with $\mathsf{U}_1$ at $a_4$.
The corresponding Riemann problem is solved by a shock $\mathsf{S}_2$ from $u_*$ to $\hat{u}_1$ travelling with negative speed $\Lambda(u_*,\hat{u}_1)$ and an undercompressive shock $\mathsf{U}_2$ from $\hat{u}_1$ to $\check{u}$.

\item $\mathsf{PT}_{1}$ interacts with the rarefaction $\mathsf{R}$ at $a_{2} \doteq (x_1,t_{a_{2}})$, where $t_{a_{2}} \doteq t_{a_{1}} + \frac{x_{1}-x_{2}}{\lambda_{1}(u_1}$.
As a result, a phase transition $\mathsf{PT}_{2}$ starts from $a_2$ and accelerates during its interaction with $\mathsf{R}$ according to the following ordinary differential equation
\begin{align*}
&\dot{x}(t)=v\bigl(\mathsf{R}\bigl(t,x(t)\bigr)\bigr),& x(t_{a_{2}})=x_{1},
\end{align*}
where, with an abuse of notation, we denoted by $\mathsf{R}(t,x)$ the value attained by the rarefaction $\mathsf{R}$ in $(t,x) \in \mathcal{C}$.

\item $\mathsf{PT}_{2}$ stops to interact with the rarefaction $\mathsf{R}$ once it reaches $a_{3} \doteq (x_{a_{3}},t_{a_{3}})$. 
Then, a phase transition $\mathsf{PT}_{3}$ from $(0,0)$ to $u_*$ and travelling with speed $v(u_*)$ starts from $a_3$.

\item $\mathsf{PT}_{3}$ interacts with $\mathsf{S}_{2}$ at $a_{5} \doteq (x_{a_{5}},t_{a_{5}})$.
The result of this interaction is a phase transition $\mathsf{PT}_{4}$ from $(0,0)$ to $\hat{u}_{1}$ travelling with speed $v(\hat{u}_{1})$.

\item $\mathsf{PT}_{4}$ interacts with $\mathsf{U}_2$ at $a_{6} \doteq (0,t_{a_{6}}) = (0, -x_1/F)$. Clearly, $t_{a_{6}}$ gives the time at which the last vehicle passes through $x=0$. Then, a contact discontinuity $\mathsf{C}_4$ from $(0,0)$ to $\check{u}$ and travelling with speed $V$ arises at $a_6$.
\end{itemize}

The simulation presented in \figurename~\ref{fig:simu01} is obtained by the explicit analysis of the wave-fronts interactions with computer-assisted computation of the interaction times and front slopes, and it corresponds to the following choice of the parameters
\begin{equation}\label{nume.data}
\begin{gathered}
a=0,\quad
R=1,\quad
\sigma=\frac{3}{10},\quad
w_{1}=-\frac{2}{5},\quad
w_{2}=\frac{3}{10},\quad
V_f = 1,\quad
x_{1}=-5,\quad
x_{2}=-1,\quad
F=\frac{3}{25}.
\end{gathered}
\end{equation}
Let us finally underline that this choice of the parameters ensures \eqref{H1} and \eqref{H2}.


\section{Technical details}\label{sec:tech}


In this section we collect the proofs concerning the properties of the Riemann solvers.

\subsection{Proofs of the main properties of \texorpdfstring{$\mathcal{R}$}{}}\label{sec:tec0}

In the following two lemmas we prove Proposition~\ref{prop:cc0}.

\begin{lemma}\label{lem:01}
The Riemann solver $\mathcal{R}$ is $\Lloc1$-continuous.
\end{lemma}

\begin{proof}
Let $(u^\varepsilon_\ell,u^\varepsilon_r) \to (u_\ell,u_r)$ in $\Omega$.
We have to prove that $\mathcal{R}[u_\ell^\varepsilon,u_r^\varepsilon] \to \mathcal{R}[u_\ell,u_r]$ in $\Lloc1$.
To do so, it is sufficient to consider the case where $\mathcal{R}[u_\ell,u_r]$ consists of just one phase transition, namely the case described in \ref{R4} of Definition~\ref{def:R}. 
Since $\mathcal{R}[u_\ell^\varepsilon,u_r^\varepsilon]$ is the juxtaposition of $\mathcal{R}[u_\ell^\varepsilon,\psi_2^-(u_r^\varepsilon)]$ and the $2$-contact discontinuity $\mathcal{R}[\psi_2^-(u_r^\varepsilon),u_r^\varepsilon]$, we are left to consider the following three cases.

\begin{itemize}[leftmargin=*]\setlength{\itemsep}{0cm}%

\item 
If $\rho_\ell = 0$, then it is sufficient to exploit the fact that $\psi_2^-(u_r^\varepsilon) \to \psi_2^-(u_r)$ to obtain that both $\Lambda(u_\ell^\varepsilon, \psi_2^-(u_r^\varepsilon))$ and $v(u_r^\varepsilon)$ converge to $v(u_r)$ and to be able to conclude.

\item 
If $\rho_\ell \ne 0$ and $u_r \ne u_p(u_\ell)$, then $w(u_r) = w_-$ and we can assume that $\mathcal{R}[u_\ell^\varepsilon,u_r^\varepsilon]$ is described by \ref{R4}.
Hence, it is sufficient to exploit the fact that $\psi_2^-(u_r^\varepsilon)\to u_r$ to obtain that $\Lambda(u_\ell^\varepsilon, \psi_2^-(u_r^\varepsilon)) \to \Lambda(u_\ell,u_r)$, $v(u_r^\varepsilon)\to v(u_r)$ and to be able to conclude.

\item 
If $\rho_\ell \ne 0$ and $u_r = u_p(u_\ell)$, then $w(u_r) = w_-$ and $\mathcal{R}[u_\ell^\varepsilon,u_r^\varepsilon]$ is described by either \ref{R4} or \ref{R4b}.
In the first case we can argue as in the previous case.
In the latter case, we exploit the fact that $u_p(u_\ell^\varepsilon) \to u_r$ and $\psi_2^-(u_r^\varepsilon)\to u_r$ to obtain that both $\Lambda(u_\ell^\varepsilon, u_p(u_\ell^\varepsilon))$ and $\lambda_1(\psi_2^-(u_r^\varepsilon))$ converge to $\Lambda(u_\ell,u_r) = \lambda_1(u_r)$ and to be able to conclude.
\qedhere
\end{itemize}
\end{proof}
%
%
%

\begin{lemma}\label{lem:02}
The Riemann solver $\mathcal{R}$ is consistent. 
\end{lemma}
\begin{proof}
Since $\mathcal{R}[u,u] = u$ for all $u \in \Omega$, it is not restrictive to assume that $u_\ell$, $u_m$ and $u_r$ are all distinct.
Furthermore, since the usual Lax Riemann solver is consistent, it is sufficient to consider the cases for which at least one phase transition is involved. 
We observe that any admissible solution performs at most one phase transition and that all the phase transitions are from $\Of^-$ to $\Oc^-$.
Moreover, $\mathcal{R}[u_\ell,u_m]$ cannot contain any contact discontinuity, since no wave can follow a contact discontinuity.
Hence, to prove \eqref{P1} or \eqref{P2} we are left to consider the cases for which $(u_\ell,u_r)$ satisfies \ref{R4} or \ref{R4b}, $\rho_\ell \ne 0$, $w(u_m) = w_-$ and the result easily follows.
\end{proof}

\subsection{Proofs of the main properties of \texorpdfstring{$\mathcal{R}_F$}{}}\label{sec:tec1}

In the next lemmas we prove Proposition~\ref{prop:cc}. 

\begin{lemma}
The Riemann solver $\mathcal{R}_F$ is $\Lloc1$-continuous.
\end{lemma}
\begin{proof}
By the $\Lloc1$-continuity of $\mathcal{R}$, it suffices to consider $(u_\ell^\varepsilon,u_r^\varepsilon) \to (u_\ell,u_r)$ with $(u_\ell^\varepsilon,u_r^\varepsilon) \in \mathcal{D}_2$ and to prove that $\mathcal{R}[u_\ell^\varepsilon,\hat{u}^\varepsilon] \to \mathcal{R}[u_\ell,u_r]$ in $\R_-$ and $\mathcal{R}[\check{u}^\varepsilon,u_r^\varepsilon] \to \mathcal{R}[u_\ell,u_r]$ in $\R_+$, where $\hat{u}^\varepsilon \doteq \hat{u}(u_\ell^\varepsilon,F)$ and $\check{u}^\varepsilon\doteq\check{u}(u_r^\varepsilon,F)$.
For notational simplicity, below we denote  $\hat{u} \doteq \hat{u}(u_\ell,F)$, $\check{u} \doteq \check{u}(u_r,F)$ and $u_* \doteq u_*(u_\ell,u_r)$.\\
We first consider the cases with $(u_\ell,u_r) \in \mathcal{D}_1$.

\begin{itemize}[itemindent=*,leftmargin=0pt]\setlength{\itemsep}{0cm}%

\item 
Assume $u_\ell, u_r\in \Of$ and $f(u_\ell) = F$.
Then, either $u_\ell \in \Of^-$ or $u_\ell \in \Of^+$.
In the first case, $\Lambda(u_\ell^\varepsilon, \hat{u}^\varepsilon) \to 0$ and $\mathcal{R}[u_\ell^\varepsilon,  \hat{u}^\varepsilon] \to u_\ell$ in $\R_-$, while in the latter $\hat{u}^\varepsilon \to u_\ell$ and $\mathcal{R}[u_\ell^\varepsilon,  \hat{u}^\varepsilon] \to \mathcal{R}[u_\ell,u_\ell] = u_\ell$.
Moreover, in both cases $ \check{u}^\varepsilon \to u_\ell$ and $\mathcal{R}[ \check{u}^\varepsilon, u_r^\varepsilon] \to \mathcal{R}[u_\ell,u_r]$.

\item 
Assume $u_\ell,u_r\in\Oc$ and $f(u_*) = F$.
Then, $\hat{u}^\varepsilon \to u_*$ and $\mathcal{R}[u_\ell^\varepsilon, \hat{u}^\varepsilon] \to \mathcal{R}[u_\ell,u_*]$.
Moreover, it is sufficient to consider the cases $\check{u}^\varepsilon \in \Of^-$ and $\check{u}^\varepsilon \in \Oc$.
In the first case $\psi_2^-(u_r^\varepsilon) \to u_*$, $\Lambda(\check{u}^\varepsilon, \psi_2^-(u_r^\varepsilon)) \to 0$ and $\mathcal{R}[\check{u}^\varepsilon, u_r^\varepsilon] \to \mathcal{R}[u_*,u_r]$ in $\R_+$, while in the latter $\check{u}^\varepsilon \to u_*$ and $\mathcal{R}[\check{u}^\varepsilon, u_r^\varepsilon] \to \mathcal{R}[u_*,u_r]$.

\item 
Assume $u_\ell\in \Oc^-$, $u_r\in \Of^-$ and $f(\psi_1(u_\ell)) = F$.
Then, $\hat{u}^\varepsilon \to \psi_1(u_\ell)$ and $\check{u}^\varepsilon = \psi_1(u_\ell)$.
As a consequence, $\mathcal{R}[u_\ell^\varepsilon, \hat{u}^\varepsilon] \to \mathcal{R}[u_\ell,\psi_1(u_\ell)]$ and $\mathcal{R}[\check{u}^\varepsilon, u_r^\varepsilon] \to \mathcal{R}[\psi_1(u_\ell),u_r]$.

\item 
Assume $u_\ell\in \Of^-$, $u_r\in \Oc^-$ and $f(u_\ell) > f(\psi_2^-(u_r)) = F$.
Then, $\hat{u}^\varepsilon = \psi_2^-(u_r)$ and therefore $\mathcal{R}[u_\ell^\varepsilon, \hat{u}^\varepsilon] \to \mathcal{R}[u_\ell,\psi_2^-(u_r)]$.
Moreover, $\Lambda(\check{u}^\varepsilon, \psi_2^-(u_r^\varepsilon)) \to 0$ and therefore $\mathcal{R}[\check{u}^\varepsilon, u_r^\varepsilon] \to \mathcal{R}[\psi_2^-(u_r),u_r]$ in $\R_+$.

\item 
Assume $u_\ell\in \Of^-$, $u_r\in \Oc^-$ and $f(\psi_2^-(u_r)) \ge f(u_\ell) = F$.
In this case, $\Lambda(u_\ell^\varepsilon, \hat{u}^\varepsilon) \to 0$ and $\mathcal{R}[u_\ell^\varepsilon, \hat{u}^\varepsilon] \to u_\ell$ in $\R_-$.
Moreover, $\check{u}^\varepsilon = u_\ell$ and $\mathcal{R}[\check{u}^\varepsilon, u_r^\varepsilon] \to \mathcal{R}[u_\ell,u_r]$.

\end{itemize}
We finally consider the cases with $(u_\ell,u_r) \in \mathcal{D}_2$.

\begin{itemize}[itemindent=*,leftmargin=0pt]\setlength{\itemsep}{0cm}%

\item 
Assume $u_\ell, u_r\in \Of$ and $f(u_\ell) > F$.
Then, $\hat{u}^\varepsilon \to \hat{u}$ and $\check{u}^\varepsilon \to \check{u}$.
As a consequence, $\mathcal{R}[u_\ell^\varepsilon, \hat{u}^\varepsilon] \to \mathcal{R}[u_\ell,\hat{u}]$ and $\mathcal{R}[\check{u}^\varepsilon, u_r^\varepsilon] \to \mathcal{R}[\check{u},u_r]$.

\item 
Assume $u_\ell,u_r\in\Oc$ and $f(u_*(u_\ell,u_r)) > F$.
Then, $\hat{u}^\varepsilon \to \hat{u}$ and therefore $\mathcal{R}[u_\ell^\varepsilon, \hat{u}^\varepsilon] \to \mathcal{R}[u_\ell,\hat{u}]$.
Moreover, it is sufficient to consider the cases $\check{u}^\varepsilon \in \Oc$ and $\check{u}^\varepsilon \in \Of^-$.
In the first case $\check{u}^\varepsilon \to \check{u}$ and $\mathcal{R}[\check{u}^\varepsilon, u_r^\varepsilon] \to \mathcal{R}[\check{u},u_r]$, while in the latter (whether $\check{u} = \psi_2^-(u_r)$ or $\check{u} = \check{u}^\varepsilon$) we have $\psi_2^-(u_r^\varepsilon) \to \psi_2^-(u_r)$ and $\mathcal{R}[\check{u}^\varepsilon, u_r^\varepsilon] \to \mathcal{R}[\check{u},u_r]$ in $\R_+$.

\item 
Assume $u_\ell\in \Oc^-$, $u_r\in \Of^-$ and $f(\psi_1(u_\ell)) > F$.
Then, $\hat{u}^\varepsilon \to \hat{u}$ and $\check{u}^\varepsilon = \check{u}$.
As a consequence, we have $\mathcal{R}[u_\ell^\varepsilon, \hat{u}^\varepsilon] \to \mathcal{R}[u_\ell,\hat{u}]$ and $\mathcal{R}[\check{u}^\varepsilon, u_r^\varepsilon] \to \mathcal{R}[\check{u},u_r]$.

\item 
Assume $u_\ell\in \Of^-$, $u_r\in \Oc^-$ and $\min\{f(u_\ell),f(\psi_2^-(u_r))\} > F$.
Then, $\hat{u}^\varepsilon = \hat{u}$ and $\check{u}^\varepsilon = \check{u}$.
As a consequence, $\mathcal{R}[u_\ell^\varepsilon, \hat{u}^\varepsilon] \to \mathcal{R}[u_\ell,\hat{u}]$ and $\mathcal{R}[\check{u}^\varepsilon, u_r^\varepsilon] \to \mathcal{R}[\check{u},u_r]$.\qedhere
\end{itemize}
\end{proof} 

\begin{lemma}\label{lem:05}
The Riemann solver $\mathcal{R}_F$ satisfies \eqref{P2} but not \eqref{P1}.
\end{lemma}

\begin{proof}
We start by proving \eqref{P2} and assume
\begin{align}\label{eq:Ziltoid01}\tag{Z1}
&\mathcal{R}_F[u_\ell,u_m](\bar x)=u_m=\mathcal{R}_F[u_m,u_r](\bar x), &\text{ for }\bar x\in \R.
\end{align}
Since by Lemma~\ref{lem:02} the Riemann solver $\mathcal{R}$ satisfies \eqref{P2}, it is not restrictive to assume that
\begin{equation}\label{eq:Ziltoid02}\tag{Z2}
\{(u_\ell,u_r), (u_\ell,u_m), (u_m,u_r)\} \cap \mathcal{D}_2 \ne \emptyset.
\end{equation}
We also observe that \eqref{eq:Ziltoid01} implies
\begin{align}\label{eq:Ziltoid03}\tag{Z3}
\mathcal{R}_F[u_\ell,u_m] \text{ does not contain any contact discontinuity},
\end{align}
because otherwise it would be not possible to juxtapose $\mathcal{R}_F[u_\ell,u_m]$ and $\mathcal{R}_F[u_m,u_r]$.
We are then left to consider the following cases.

\begin{itemize}[itemindent=*,leftmargin=0pt]\setlength{\itemsep}{0cm}%

\item
Assume $u_\ell, u_m \in \Of$.
In this case, by \eqref{eq:Ziltoid03} we have $(u_\ell,u_m) \in \mathcal{D}_2$ and $u_m = \check{u}(u_m,F)$, namely $f(u_\ell) > F = f(u_m)$.
Then, by \eqref{eq:Ziltoid01} we have that either $u_m \in \Of^-$ and $f(\psi_2^-(u_r)) > F$ or $u_m \in \Of^+$ and $u_r \in \Of$.
In both cases it is easy to conclude.

\item
Assume $u_\ell, u_m \in \Oc$.
In this case, by \eqref{eq:Ziltoid03} we have either $(u_\ell,u_m) \in \mathcal{D}_2$ or $(u_\ell,u_m) \in \mathcal{D}_1$ and $w(u_\ell) = w(u_m)$.
In the first case, whether $\check{u}(u_m,F) \in \Oc$, $u_m = \check{u}(u_m,F)$ and $w(u_m) < w(u_\ell)$ or $\check{u}(u_m,F) \in \Of^-$, $w(u_m) = w_- \le w(u_\ell)$ and $f(u_m) > F$, by \eqref{eq:Ziltoid01} we have that $v(u_r) = v(u_m)$.
In the latter case, by \eqref{eq:Ziltoid01} and \eqref{eq:Ziltoid02} we have that $f(u_m) = F$, $v(u_r) > v(u_m)$ and $(u_m,u_r) \in \mathcal{D}_2$.
In both cases it is easy to conclude.

\item Assume $u_\ell\in \Oc^-$ and $u_m \in \Of^-$. 
In this case, by \eqref{eq:Ziltoid03} we have $(u_\ell,u_m)\in\mathcal{D}_2$ and $u_m=\check{u}(u_m,F)$.
Then, by \eqref{eq:Ziltoid01} we have that $u_r$ has to satisfy $f(\psi_2^-(u_r))>F$.
Hence, it is easy to conclude.

\item Assume $u_\ell\in \Of^-$ and $u_m \in \Oc^-$.
In this case, by \eqref{eq:Ziltoid03} we have $w(u_m) = w_-$.
Moreover, by \eqref{eq:Ziltoid02} and \eqref{eq:Ziltoid03} we have either $f(u_m) = F < f(u_\ell)$  and $v(u_r) > v(u_m)$ or $f(u_m) > F$ and $v(u_r) = v(u_m)$.
In both cases it is easy to conclude. 

\end{itemize}
Finally, it remains to show that $\mathcal{R}_F$ does not satisfy \eqref{P1}. For example, take $u_\ell\in \Of^-$ and $u_r\in \Oc^-$ such that $f(u_\ell) < F < f(\psi_2^-(u_r))$.
Moreover, fix $\bar{x}>0$ so that $\mathcal{R}_F[u_\ell,u_r](\bar x) = \psi_2^-(u_r)$ and take $u_m \doteq \psi_2^-(u_r)$.
Observe that $(u_m, u_r)\in \mathcal{D}_2$, $(u_\ell, u_r),\,(u_\ell,u_m)\in\mathcal{D}_1$ and \eqref{P1} does not hold true.
\end{proof}

Finally, we accomplish the proof of Proposition ~\ref{prop:IDR} on the minimal invariant domains for $\mathcal{R}_F$.
We remark that $(F/V,Q(F/V))$ is the point of intersection between the lines $\{u\in\Omega \,:\, f(u)=F\}$ and $\Of$: if $F \ge V\sigma_-$, this point belongs to the region $\Oc$, otherwise it is in $\Of^-$. 

\begin{proof}[Proof of Proposition~\ref{prop:IDR}]
%
Let us first prove \ref{I11}.
The invariance of $\mathcal{I}_f$ is an easy consequence of Definition~\ref{def:01}, hence we are left to prove the minimality of $\mathcal{I}_f$.
Let $\mathcal{I}$ be an invariant domain for $\mathcal{R}_F$ containing $\Of$.
Then, $\mathcal{I}$ has to contain
\[
\mathcal{R}_F[ \{ (u_\ell,u_r) \in \Of^2 \,:\, f(u_\ell) > F\} ](\R)
= \Of \cup \{u\in\Oc \,:\, f(u)= F\} \cup \mathcal{I}_2.
\]
As a consequence, $\mathcal{I}$ has to contain also
\begin{align*}
&\mathcal{R}_F[ \{ (u_\ell,u_r) \in \Oc^2 \,:\, f(u_\ell) = f(u_r) = F,\ v(u_\ell) > v(u_r)\} ](\R)
=
\begin{cases}
\mathcal{I}_1 &\text{if }F \le V\,\sigma_-,
\\
\{u \in \mathcal{I}_1 \,:\, f(\psi_1(u)) \ge F\} &\text{if }F > V\,\sigma_-.
\end{cases}
\end{align*}
Finally, if $F > V\,\sigma_-$, then $\mathcal{I}$ has to contain also
\[
\mathcal{R}_F[ \{ (u_\ell,u_r) \in \Of^+ \times \Oc \,:\, f(u_\ell) \le F = f(u_r)\} ](\R)
=
\{u \in \mathcal{I}_1 \,:\, f(\psi_1(u)) \le F\}.
\]
In conclusion we proved that $\mathcal{I} \supseteq \mathcal{I}_f$.
%
%
%

Now, to prove \ref{I12} it is sufficient to observe that by Definition~\ref{def:01} we have that there exist $u_\ell,u_r\in \Oc$ such that the values attained by $\mathcal{R}_F[u_\ell,u_r]$ exit $\Oc$ if and only if $F<V\sigma_-$, and in this case
\[
\mathcal{R}_F[\Oc^2](\R)
\setminus \Oc
=
\mathcal{R}_F[\{ (u_\ell,u_r) \in \Oc^2 \,:\, f(\psi_2^-(u_r)) > F\} ](\R)
\setminus \Oc
=\{(F/V,Q(F/V))\}
\subset \Of^-.\qedhere
\]
\end{proof}

\subsection{Proofs of the main properties of \texorpdfstring{$\mathcal{S}$}{}}\label{sec:tec2}

The following lemmas contain the proof of Proposition~\ref{prop:cc2}. We recall that the Riemann solver for the PT$^p$ model has already been studied in \cite{BenyahiaRosini01}.

\begin{lemma}
The Riemann solver $\mathcal{S}$ is $\Lloc1$-continuous.
\end{lemma}

\begin{proof}
Given the analysis of Lemma~\ref{lem:01} and Remark~\ref{rem:specialRS}, it is sufficient to consider only the case in which $\mathcal{S}[u_\ell,u_r]$ is a single phase transition with $(u_\ell,u_r)$ satisfying one of the conditions \eqref{eq:specialRS1},\eqref{eq:specialRS2},\eqref{eq:specialRS3}.
Hence, let $(u^\varepsilon_\ell,u^\varepsilon_r)\to (u_\ell,u_r)$ in $\Omega$ and consider the following cases.

\begin{itemize}[leftmargin=*]\setlength{\itemsep}{0cm}%

\item
Assume that $(u_\ell,u_r)$ satisfies \eqref{eq:specialRS1} with $u_r=\psi_1^f(u_\ell)$ and $v(u_\ell) = V_c$.
In this case, it is sufficient to exploit the fact that $\psi_1^c(u_\ell^\varepsilon) \to u_\ell$ and $\psi_1^f(u_\ell^\varepsilon) \to u_r$ to obtain that both $\lambda_1(u_\ell^\varepsilon)$ and $\lambda_1(\psi_1^c(u_\ell^\varepsilon))$ converge to $\lambda_1(u_\ell)$,  $\Lambda(\psi_1^c(u_\ell^\varepsilon),\psi_1^f(u_\ell^\varepsilon)) \to \Lambda(u_\ell,u_r)$ and to be able to conclude.

\item
Assume that $(u_\ell,u_r)$ satisfies \eqref{eq:specialRS2} with $u_\ell=\psi_1^f(u_r)$ and $v(u_r) = V_c$.
In this case, it is sufficient to exploit the fact that $\psi_1^c(u_\ell^\varepsilon) \to u_r$ and $u_*(u_\ell^\varepsilon,u_r^\varepsilon) \to u_r$ to obtain that both $\lambda_1(\psi_1^c(u_\ell^\varepsilon))$ and $\lambda_1(u_*(u_\ell^\varepsilon,u_r^\varepsilon))$ converge to $\lambda_1(u_r)$, $\Lambda(u_\ell^\varepsilon,\psi_1^c(u_\ell^\varepsilon)) \to \Lambda(u_\ell,u_r)$, $\Lambda(u_*(u_\ell^\varepsilon,u_r^\varepsilon),u_r^\varepsilon) \to V_c$ and to be able to conclude.

\item
Assume that $(u_\ell,u_r)$ satisfies \eqref{eq:specialRS3} with $u_r = u_-^c$.
In this case, it is sufficient to exploit the fact that $\psi_2^-(u_r^\varepsilon) \to u_r$ to obtain that $\lambda_1(\psi_2^-(u_r^\varepsilon)) \to \lambda_1(u_r)$, $\Lambda(u_\ell^\varepsilon,u_r) \to \Lambda(u_\ell,u_r)$, $\Lambda(\psi_2^-(u_r^\varepsilon),u_r^\varepsilon) \to V_c$ and to be able to conclude.
\qedhere
\end{itemize}
\end{proof}

\begin{lemma}\label{lem:08}
The Rieman solver $\mathcal{S}$ is consistent.
\end{lemma}
\begin{proof}
Since $\mathcal{S}[u,u] = u$ for all $u \in \Omega$, it is not restrictive to assume that $u_\ell$, $u_m$ and $u_r$ are all distinct.
Furthermore, $\mathcal{S}[u_\ell,u_m]$ cannot contain any contact discontinuity, since no wave can follow a contact discontinuity.
By Lemma~\ref{lem:02} and by Remark~\ref{rem:specialRS}, to prove \eqref{P1} or \eqref{P2} we are left to consider the cases for which $(u_\ell,u_r)$ satisfies \eqref{eq:specialRS1} with $u_m$ such that $w(u_m) = w(u_\ell)$ and $v(u_m) > v(u_\ell)$, or \eqref{eq:specialRS2} with $u_m$ such that $w(u_m) = w(u_\ell)$ and $v(u_m) \ge v(u_r)$, or \eqref{eq:specialRS3} with $u_m$ such that $w(u_m) = w_-$ and $v(u_m) \ge v(u_r)$. 
Hence, the result easily follows.
\end{proof}

\subsection{Proofs of the main properties of \texorpdfstring{$\mathcal{S}_F$}{}}\label{sec:tec3}

In this final section, we accomplish the proof of Proposition~\ref{prop:cc3}. Recall that the same constrained Riemann solver for the PT$^p$ model has already been studied in \cite{BenyahiaRosini02}.

\begin{example}
The Riemann solver $\mathcal{S}_F$ is not $\Lloc1$-continuous. Indeed, take $F > f(u_-^c)$ and consider $u_\ell, u_r, u_\ell^\varepsilon \in \Of$ with $f(u_\ell) = F < f(u_\ell^\varepsilon)$ and $u_\ell^\varepsilon\to u_\ell$.
In this case $\mathcal{S}_F[u_\ell^\varepsilon,u_r]$ does not converge to $\mathcal{S}_F[u_\ell,u_r]$ in $\Lloc1$, since $\mathcal{S}_F[u_\ell,u_r] = u_\ell$ in $\R_-$ and the restriction of $\mathcal{S}_F[u_\ell^\varepsilon,u_r]$ to $\R_-$ converges to
\[
\begin{cases}
u_\ell &\text{if } x<\Lambda(u_\ell,u_\#),\\
u_\# &\text{if } \Lambda(u_\ell,u_\#)<x<0,
\end{cases}
\]
where $u_\# = u^c_-$ if $u_\ell \in \Of^-$ and $u_\# = \psi_1^c(u_\ell)$ if $u_\ell \in \Of^+$.
\end{example}

\begin{lemma}
The Riemann solver $\mathcal{S}_F$ satisfies \eqref{P2} but not \eqref{P1}.
\end{lemma}

\begin{proof}
Assume that $\mathcal{S}_F[u_\ell,u_m](\bar x)=u_m=\mathcal{S}_F[u_m,u_r](\bar x)$, for some $\bar x\in\R$.
Since by Lemma~\ref{lem:08} and Lemma~\ref{lem:05} the Riemann solvers $\mathcal{S}$ and $\mathcal{R}_F$ already satisfy \eqref{P2}, by Remark~\ref{rem:special} we have to consider the cases where at least one among $(u_\ell,u_r)$, $(u_\ell,u_m)$ and $(u_m,u_r)$ belong to $\mathcal{D}_2$ and satisfy one of the conditions \eqref{eq:special1},\eqref{eq:special2},\eqref{eq:special3}. 
We observe that $\mathcal{S}_F[u_\ell,u_m]$ cannot present any contact discontinuity, otherwise it would not be possible to juxtapose $\mathcal{S}_F[u_\ell,u_m]$ and $\mathcal{S}_F[u_m,u_r]$.
Hence, we are left to consider $(u_\ell,u_r) \in \mathcal{D}_2$ satisfying \eqref{eq:special1} with $u_m \in \{\hat{u},\check{u}\}$, or \eqref{eq:special2} with $u_m \in \{\hat{u},\check{u}\}$, or \eqref{eq:special3} with $u_m \in \{\check{u}\}\cup\mathcal{S}_F[u_\ell,\hat{u}](\R_-)$.
In all these cases, it is easy to see that \eqref{P2} holds true.

Finally, by the example of Lemma \ref{lem:05} we have that $\mathcal{S}_F$ does not satisfy \eqref{P1}.
\end{proof}


\section*{Acknowledgements} 


MDR thanks Rinaldo M.\ Colombo and Paola Goatin for useful discussions.



\end{document}